%% file: manuscript.tex
\documentclass[a4paper]{amsart}

\usepackage[UKenglish]{babel}
\usepackage[T1]{fontenc}
\usepackage{lmodern}
\usepackage{microtype}
\usepackage{amsmath,amsthm,amsfonts,amssymb}
\usepackage{mathtools}
\usepackage{centernot}
\usepackage{graphicx}
\usepackage{float}
\usepackage{tikz}
\usepackage{enumitem}
\usepackage{booktabs}
\usepackage{xcolor}
\usepackage{hyperref}
\usepackage{cancel}

\usepackage[capitalise,nameinlink,noabbrev]{cleveref}

\hypersetup{
    colorlinks=true,
    linkcolor=blue!50!black,
    citecolor=blue!50!black,
    urlcolor=blue!50!black
}

\newtheorem{theorem}{Theorem}[section]
\newtheorem{corollary}[theorem]{Corollary}
\newtheorem{lemma}[theorem]{Lemma}
\newtheorem{definition}[theorem]{Definition}
\newtheorem{question}[theorem]{Question}
\newtheorem{proposition}[theorem]{Proposition}
\newtheorem{remark}[theorem]{Remark}

\newcommand{\K}{\mathfrak{K}}
\newcommand{\A}{\mathcal{A}}

\newcommand{\TLH}{\mathsf{TLH}}
\newcommand{\baire}{\omega^{<\omega}}
\newcommand{\Baire}{\omega^\omega}
\newcommand{\Cantor}{2^{\omega}}
\newcommand{\cantor}{2^{<\omega}}
\newcommand{\define}[1]{\textbf{#1}}
\newcommand{\learnerM}{\mathbf{M}}
\newcommand{\ex}{\mathbf{Ex}}
\newcommand{\bc}{\mathbf{BC}}
\newcommand{\Learn}{\mathrm{Learn}}
\newcommand{\ce}{\mathrm{ce}}
\newcommand{\Id}{\mathrm{Id}}
\newcommand{\learnreducible}{\leq_{\Learn}}

\newcommand{\nEce}[1]{\mathrel{\cancel{{#1}^{\ce}}}}

\allowdisplaybreaks

\title[Tolerant learning hierarchy for c.e.\ sets]{The Tolerant Learning Hierarchy for c.e.\ Sets}

\author[V. Cipriani]{Vittorio Cipriani}
\address{Institute of Discrete Mathematics and Geometry, TU Wien, Vienna, Austria}
\email{\href{mailto:vittorio.cipriani17@gmail.com}{vittorio.cipriani17@gmail.com}}

\author[M. Ritter]{Martin Ritter}
\address{Institute of Discrete Mathematics and Geometry, TU Wien, Vienna, Austria}
\email{\href{mailto:martin.ritter@tuwien.ac.at}{martin.ritter@tuwien.ac.at}}

\author[L. San Mauro]{Luca San Mauro}
\address{Department of Humanistic Research and Innovation, University of Bari, Bari, Italy}
\email{\href{mailto:lucafrancesco.sanmauro@uniba.it}{lucafrancesco.sanmauro@uniba.it}}



\thanks{The first and second author's research was partially supported by the Austrian Science Fund (FWF)
10.55776/P36781. The third author is a member of INDAM--GNSAGA}

\begin{document}

\begin{abstract}
Learning from positive data models a scenario in which a learner receives longer and longer initial segments of an enumeration of an unknown set and repeatedly outputs a conjecture about it. Classical identification in the limit requires the learner to eventually stabilise to a hypothesis naming the target exactly, while anomalous learning relaxes this requirement by allowing finitely many errors in the final hypothesis. We develop a natural generalisation for computably enumerable (c.e.) sets in which correctness is determined by an arbitrary equivalence relation on their indices that respects equality of the sets they enumerate. We call the resulting framework tolerant learning. Exact learning and finite-anomaly learning arise as special cases. By comparing tolerance relations according to the families they make learnable, we obtain a new degree structure, the tolerant learning hierarchy.

We extend the classical methods of locking sequences and tell-tales to this setting and derive a combinatorial characterisation of those tolerance relations that are universal (that is, under which every family of nonempty computably enumerable sets is learnable). We then show that the hierarchy is a bounded lower semilattice but not an upper semilattice, is atomless and coatomless, and contains chains and antichains of size continuum. We also analyse natural benchmark relations, including c.e.-index analogues of standard equivalence relations from descriptive set theory, and establish substantial incomparability among their learning powers. Finally, we separate non-effective from computable tolerant learning by exhibiting a universal relation that is not computably universal. These results show that semantic tolerance gives rise to a genuine degree theory at the interface of algorithmic learning and the theory of equivalence relations.
\end{abstract}
\maketitle

\section{Introduction}

Algorithmic learning theory asks when an idealised learner can identify an unknown object from a growing stream of information.  In Gold's model of identification in the limit \cite{Gold67}, the learner receives longer and longer finite pieces of data, may revise its conjecture finitely many times, and succeeds if its hypotheses eventually stabilise to a correct one.  Putnam's trial-and-error predicates \cite{putnam1965trial} provide a closely related logical viewpoint on eventual convergence to truth.  Standard background can be found in \cite{jain1999systems}.

Two classes of targets have dominated the classical theory: computable functions and computably enumerable sets, or languages.  Computable functions are central in prediction and program-synthesis interpretations of inductive inference, already in the work of Blum and Blum \cite{blum1975toward}.  Computably enumerable (c.e.) sets are the natural targets for learning from \emph{positive} data: a text for a language presents examples from the language and never certifies non-membership.  This is the setting of Gold's language-learning model and Angluin's tell-tale theorem \cite{Gold67,angluin1980inductive}; for a survey of learning indexed families from positive data, see \cite{LangeZeugmannZilles2008Survey}.  We work throughout with c.e.\ sets, viewed via a fixed standard enumeration $(W_e)_{e\in\omega}$.

The classical success criterion is exact identification: eventually the learner must output an index for precisely the target set.  This is the usual $\ex$ (explanatory) convergence requirement; by contrast, $\bc$ (behaviourally correct) learning asks only that, from some point on, all hypotheses be semantically correct, possibly without stabilising to one index.  The tolerant criterion introduced below retains the syntactic stabilisation requirement of $\ex$-learning and varies only the semantic criterion of correctness.  A well-established weakening is \emph{anomalous learning}, where finitely many errors are allowed in the final hypothesis.  In the standard anomalous notation, $\ex^n$ and $\bc^n$ allow at most $n$ anomalies, while $\ex^*$ and $\bc^*$ allow an arbitrary finite number.  For computable functions, Case and Smith \cite{case1978anomaly,case1983comparison} developed the comparison theory of these criteria, including the anomaly hierarchy and the role of $\bc^*$.  For language learning from positive data, approximate correctness was already investigated by Osherson and Weinstein \cite{osherson1982criteria}; finite-anomaly learning, succinctness phenomena for anomalous learners, and robust variants were studied further; see, for example, \cite{case1982machine,caseJainSharma1996Anomalous,jain1999systems,JainStephan2003TourRobust,LangeGrieserZeugmann2005Approx}.  Approximation criteria allowing controlled infinite error sets were also studied in \cite{royer1986Approximations,smithVelauthapillai1990ApproximatePrograms}.  Earlier approximate perspectives go back at least to Feldman \cite{feldman1972some} and Wharton \cite{wharton1974approximate}.

The point of departure of the present paper is that finite symmetric difference is only one possible semantics of approximation.  Once hypotheses are indices for c.e.\ sets, any equivalence relation $E^{\ce}$ on indices can be used as a tolerance relation: a final conjecture $i$ is correct for a target indexed by $e$ exactly when $i\mathrel{E^{\ce}}e$.  We require $E^{\ce}$ to extend equality of c.e.\ sets, denoted $=^{\ce}$, so that the criterion depends only on the c.e.\ sets named by indices.  Exact learning is obtained from $=^{\ce}$, and $\ex^*$-learning is obtained from the relation $E^{\ce}_0$ of finite symmetric difference.  This leads to the basic notion studied here: learning c.e.\ sets \emph{tolerantly}, that is, up to a prescribed equivalence relation on indices.

This formulation is useful for two reasons.  First, it separates the learning mechanism from the semantics of correctness.  The same stream of positive data may support different identification tasks depending on whether final hypotheses are compared by equality, by finite error, by columnwise finite error, by equality of sets of columns, or by some other invariant (formal definitions of these equivalence relations are provided later).  Secondly, it gives a natural way to compare tolerance notions.  We write
\[
E^{\ce}\leq_{\Learn}F^{\ce}
\]
when every family of nonempty c.e.\ sets learnable up to $E^{\ce}$ is also learnable up to $F^{\ce}$.  Modulo mutual reducibility this preorder yields the \define{tolerant learning hierarchy}
\[
\TLH=(\mathcal E^{\ce},\leq_{\Learn}).
\]
Thus the objects of the hierarchy are not target families, but semantic criteria for eventual correctness.

A major theme of the paper is that the natural benchmarks for this hierarchy are not arbitrary.  Several are c.e.\ analogues of the fundamental equivalence relations that organise descriptive set theory.  The relation $E_0$ of eventual agreement on reals becomes, in the c.e.\ setting, equality modulo finite symmetric difference.  The relations $E_1$, $E_3$, and $E_{\mathrm{set}}$ have column-based c.e.\ counterparts: eventual equality of columns, columnwise finite difference, and equality of the set of c.e.\ columns.  Their classical counterparts are standard reference points under Borel and continuous reducibility, with $E_0$ playing its canonical role in the Glimm--Effros dichotomy \cite{harringtonKechrisLouveau1990GlimmEffros}, and they also occur in the descriptive-set-theoretic approach to learning algebraic structures, where $E$-learnability calibrates a family by reducing its isomorphism relation to a benchmark equivalence relation \cite{BazhenovCiprianiSanMauro2023Borel,BazhenovCiprianiJainSanMauroStephan2026Classifying}.  Treating their c.e.\ versions as benchmark tolerance relations lets us ask whether the learning order reflects the familiar descriptive-set-theoretic order or instead produces genuinely new comparisons.  One of our conclusions is that it produces a genuinely new order.

Although our targets are c.e.\ sets rather than algebraic structures, the present perspective is informed by that broader programme.  Learning of structures from text or informant has been developed for specific algebraic classes and for general logical characterisations \cite{StephanVentsov2001LearningText,MerkleStephan2004Trees,HarizanovStephan2007VectorSpaces,FokinaKotzingSanMauro2019EquivalenceStructures,BazhenovFokinaSanMauro2020Informant,BazhenovFokinaRosseggerSoskovaVatev2024Text}.  Later work measured the Turing and mind-change complexity of such learning and used Borel equivalence relations as benchmarks for learning strength \cite{BazhenovSanMauro2021TuringComplexity,BazhenovCiprianiSanMauro2022MindChange,BazhenovCiprianiSanMauro2023Borel,BazhenovCiprianiJainSanMauroStephan2026Classifying}.  Here we move the comparison to c.e.\ indices themselves: the semantic equivalence relation is not an external benchmark for a class of structures, but the correctness criterion governing convergence of hypotheses.

There is also a close but distinct connection with computable reducibility on equivalence relations on $\omega$.  Important c.e.-index relations were analysed by Fokina, Friedman, and Nies \cite{FokinaFriedmanNies2012Sigma3}, and the computable-reducibility hierarchy has been studied extensively; see, for example, \cite{gao2001computably,AndrewsBadaevSorbi2017Survey,coskey2012hierarchy,andrews2023structure,andrews24FSjump}.  Our comparison is different in kind.  Computable reducibility asks for a uniform computable translation of classifications.  Learn-reducibility asks whether one correctness semantics makes at least as many positive-data learning problems solvable as another.  Neither relation should be expected to determine the other.

The closest prior work of which we are aware is \cite{BelangerGaoJainLiStephan2023}, where positive equivalence relations are used to study classical learnability notions for one-one numbered c.e.\ families closed under the relation.  Learning equivalence relations themselves has also recently been studied in a Polish-space setting \cite{RosseggerSlamanSteifer2025Polish}.  Our setting is broader in two respects.  We allow arbitrary equivalence relations on c.e.\ indices extending $=^{\ce}$, and our main object is the global preorder of all such relations induced by learnability rather than the behaviour of a fixed family under a fixed relation.

At first sight one might expect the hierarchy to be tame.  Positive-data learning is constrained by a compactness phenomenon: if every finite fragment of a target can be extended to a wrong proper subtarget, no learner can know when to stop revising.  Weakening correctness might seem either too mild to matter or so coarse that the theory quickly collapses.  Our results show that neither expectation is right.  The classical obstruction persists, but the space of possible tolerance semantics has a rich internal degree structure.

We now describe the main results.

\begin{enumerate}
    \item We generalise \emph{locking sequences} \cite{blum1975toward} to arbitrary $E^{\ce}$ and prove the corresponding tell-tale theorem.  A family is $E^{\ce}$-learnable exactly when every target has a finite tell-tale preventing wrong proper subtargets outside its $E^{\ce}$-class.  We then dualise tell-tales into \emph{absorbing sets}, which package the obstruction to learnability in a form well suited for structural arguments.  This yields a useful bridge theorem: if $E^{\ce}<_{\Learn}F^{\ce}$, then there is a non-$E^{\ce}$-learnable family all of whose members lie in a single $F^{\ce}$-class.
\item A relation is \emph{universal} if every family of nonempty c.e.\ sets is learnable up to it.  We characterise universality by a simple locking property for increasing finite approximations: whenever finite sets increase to a nonempty c.e.\ set $W_e$, their indices must eventually fall into the $E^{\ce}$-class of $e$.  This theorem clarifies exactly how coarse a tolerance relation must be in order to defeat the usual positive-data obstruction.
\item The tolerant learning hierarchy is bounded, with bottom $[=^{\ce}]_{\equiv_{\mathrm{Learn}}}$ and top $[\Id_1]_{\equiv_{\mathrm{Learn}}}$.  It is a lower semilattice, and meets are represented by intersections of equivalence relations.  It is not an upper semilattice: indeed, we construct two degrees whose common upper bounds have no minimal element.  The hierarchy is also atomless and coatomless, and both its width and its height are $2^{\aleph_0}$.  The proofs use absorbing witnesses to embed large combinatorial orders into principal initial and final segments of $\TLH$.
\item We locate the c.e.\ analogues of the Borel benchmarks $E_0$, $E_1$, $E_3$, and $E_{\mathrm{set}}$, as well as the relations $E^{\ce}_{\min}$ and $E^{\ce}_{\max}$.  Finite-error tolerance $E^{\ce}_0$ is strictly stronger than exact learning, but it is weak in the sense that no nontrivial quotient of it is universal.  We show that $E^{\ce}_1$, $E^{\ce}_3$, $E^{\ce}_{\mathrm{set}}$, and $E^{\ce}_{\max}$ are pairwise incomparable under $\leq_{\Learn}$, while $E^{\ce}_{\min}$ is universal.  Thus the c.e.\ versions of familiar descriptive-set-theoretic relations do not line up as they do under standard reducibilities.
\item Sections~\ref{sect:prelim}--\ref{sec:comparison} impose no computability condition on learners.  In the final section we revisit universality for computable learners and show that non-effective universality is strictly weaker: if $K$ is the halting set, then the two-class relation $E^{\ce}_K$ induced by containment in $K$ is universal but not computably universal.
\end{enumerate}

\smallskip

The paper is organised as follows.  \Cref{sect:prelim} fixes notation and defines tolerant learning and learn-reducibility.  \Cref{section:characterizations} develops locking sequences, tell-tales, absorbing sets, and universal relations.  \Cref{section:structural_properties} proves the main structural theorems about $\TLH$.  \Cref{sec:comparison} analyses the benchmark relations and contrasts learn-reducibility with Borel, continuous, and computable reducibility.  \Cref{section:effective_learners} separates universal from computably universal tolerance relations.
\section{Preliminaries}\label{sect:prelim}

We denote by $\omega$ the set of natural numbers $\{0,1,2,\dots\}$. Given $A\subseteq\omega$ and $n\in\omega$, we write
\[
A\restriction n:=A\cap\{0,\dots,n-1\}.
\]
We fix a computable bijective pairing function $\langle\cdot,\cdot\rangle\colon\omega^2\to\omega$.

We denote by $\baire$ and $\Baire$, respectively, the sets of all finite and infinite sequences of natural numbers. Given $\sigma\in\baire$, let $|\sigma|$ be its length and, for $n\leq|\sigma|$, let $\sigma[n]$ be the string consisting of $\sigma(0),\dots,\sigma(n-1)$. The concatenation of finite strings $\sigma,\tau\in\baire$ is denoted by $\sigma^\smallfrown\tau$, although we often simply write $\sigma\tau$. If $f\in\Baire$, then $\sigma^\smallfrown f$ denotes the corresponding infinite concatenation. We define
\[
\operatorname{range}(\sigma):=\{\sigma(n):n<|\sigma|\}.
\]
For $\sigma,\tau\in\baire$, we write $\sigma\sqsubseteq\tau$ if $\sigma$ is an initial segment of $\tau$, and $\sigma\sqsubset\tau$ if it is a proper initial segment. If neither string extends the other, we say that they are \define{incomparable} and write $\sigma\mid\tau$. For $f\in\Baire$ and $\sigma\in\baire$, we write $\sigma\sqsubseteq f$ (equivalently, $f\sqsupseteq\sigma$) if $f[|\sigma|]=\sigma$.

A \define{tree} is a nonempty set $T\subseteq\baire$ closed under initial segments. It is a \define{binary tree} if $T\subseteq\cantor$. A real $f\in\Baire$ is a \define{path} through $T$ if $f[n]\in T$ for every $n\in\omega$. The \define{body} of $T$, denoted by $[T]$, is the set of all paths through $T$.

An \define{enumeration} of a set $A$ is a surjection $\mu\colon\omega\to A$. As for finite strings, $\mu[n]$ denotes the string $\mu(0),\dots,\mu(n-1)$.

For an equivalence relation $E$ on a set $X$ and an element $x\in X$, we write $[x]_E:=\{y\in X:yEx\}$.

A \define{preorder} on a set $P$ is a reflexive and transitive binary relation $\leq$ on $P$. A \define{partially ordered set}, or \define{poset}, is a pair $(P,\leq)$ in which $\leq$ is also antisymmetric. Given $a,b\in P$, we write $a<b$ if $a\leq b$ and $b\nleq a$, and we say that $a$ and $b$ are \define{incomparable} if $a\nleq b$ and $b\nleq a$. A set $C\subseteq P$ is a \define{chain} if any two of its elements are comparable, and an \define{antichain} if any two distinct elements are incomparable. The \define{height} of $P$ is the supremum of the cardinalities of its chains, and its \define{width} is the supremum of the cardinalities of its antichains.

A poset is \define{bounded} if it has a least and a greatest element, called its \define{bottom} and \define{top}, respectively. If $P$ has bottom $\bot$, an element $a\in P$ is an \define{atom} if $\bot<a$ and there is no $b\in P$ with $\bot<b<a$. Dually, if $P$ has top $\top$, an element $a\in P$ is a \define{coatom} if $a<\top$ and there is no $b\in P$ with $a<b<\top$. A poset is \define{atomless} if it has no atoms and \define{coatomless} if it has no coatoms. It is \define{dense} if, whenever $a<b$, there is some $x$ with $a<x<b$.

A poset $(P,\leq)$ is a \define{lower semilattice} if any two elements $a,b\in P$ have a \define{meet}, that is, a greatest lower bound, denoted by $a\wedge b$. It is an \define{upper semilattice} if any two elements have a \define{join}, that is, a least upper bound, denoted by $a\vee b$. A poset that is both a lower and an upper semilattice is a \define{lattice}.

For any set $X$, we denote by $\mathcal P_{\mathrm{fin}}(X)$ the set of all finite subsets of $X$. We write $(\forall^\infty n)\,\varphi(n)$ to mean that $\varphi(n)$ holds for all but finitely many $n$.

\subsection{The learning paradigm}
\label{sec:learningparadigm}

A set $A\subseteq\omega$ is \define{computably enumerable} (\define{c.e.}) if it is the domain of a partial computable function. We fix a standard enumeration $(\varphi_e)_{e\in\omega}$ of the partial computable functions and put $W_e:=\operatorname{dom}(\varphi_e)$. Thus $(W_e)_{e\in\omega}$ is a standard enumeration of the c.e.\ sets, and distinct indices may enumerate the same set. We write
\[
e=^{\mathrm{ce}}e'\quad\Longleftrightarrow\quad W_e=W_{e'}.
\]
Unless explicitly stated otherwise, every equivalence relation denoted by $E^{\mathrm{ce}}$, $F^{\mathrm{ce}}$, and so on is a relation on $\omega$ extending $=^{\mathrm{ce}}$. No definability or effectivity assumption is imposed on such relations. Since these relations extend $=^{\mathrm{ce}}$, they induce well-defined equivalence relations on c.e.\ sets; accordingly, for c.e.\ sets $A$ and $B$ we may write $A\mathrel{E^{\mathrm{ce}}}B$ to mean that some (equivalently, every) pair of indices for $A$ and $B$ is $E^{\mathrm{ce}}$-equivalent.

In the following definition, and throughout Sections~\ref{sect:prelim}--\ref{sec:comparison}, a \define{learner} is simply a function from $\omega^{<\omega}$ to $\omega$. No effectivity assumption is imposed until Section~\ref{section:effective_learners}.

\begin{definition}
\label{def:learningparadigm}
Let $\K$ be a family of nonempty c.e.\ sets, and let $E^{\mathrm{ce}}$ be an equivalence relation on $\omega$ extending $=^{\mathrm{ce}}$. We say that $\K$ is \define{$E^{\mathrm{ce}}$-learnable} if there is a learner $\learnerM\colon\omega^{<\omega}\to\omega$ such that, for every $e\in\omega$ with $W_e\in\K$ and every enumeration $\mu$ of $W_e$, there is some $i\in\omega$ such that
\[
\lim_{s\to\infty}\learnerM(\mu[s])=i
\quad\text{and}\quad
i\mathrel{E^{\mathrm{ce}}}e.
\]
\end{definition}

Here the displayed limit means eventual constancy. Since $E^{\mathrm{ce}}$ extends $=^{\mathrm{ce}}$, the definition depends only on the family of c.e.\ sets and not on any choice of indices for its members. Throughout the paper, target families are understood to consist of nonempty sets unless explicitly stated otherwise; the ambient enumeration $(W_e)_{e\in\omega}$, the hypotheses output by learners, and the equivalence relations on indices may still involve the empty set.

The relation $=^{\mathrm{ce}}$ yields classical identification in the limit for families of c.e.\ sets; see \cite[Definition~3.12]{jain1999systems}. Another fundamental example is $E^{\mathrm{ce}}_0$, defined by
\[
e\mathrel{E^{\mathrm{ce}}_0}e'
\quad\Longleftrightarrow\quad
|W_e\mathbin{\triangle}W_{e'}|<\infty,
\]
where $\triangle$ denotes the symmetric difference.
Finally, $\Id_1$ denotes the one-class equivalence relation on $\omega$.

\begin{remark}
In \cite[Definition~3.12]{jain1999systems}, learners receive a \define{text}: an infinite sequence in $\omega\cup\{\sharp\}$ whose numerical content is the target set, with $\sharp$ representing a pause in the presentation. For nonempty targets, this model is equivalent to the enumeration model used here.

Indeed, every enumeration is a text without occurrences of $\sharp$. Conversely, let $T$ be a text for a nonempty set $A$, and let $s_0$ be the first stage at which a numerical symbol appears. Before stage $s_0$, a learner may output an arbitrary hypothesis. From stage $s_0$ onward, delete the initial block of $\sharp$ symbols and replace each subsequent occurrence of $\sharp$ by the most recently seen numerical symbol. The resulting infinite sequence is an enumeration of $A$, so a learner for enumerations can be simulated on texts. The converse simulation is immediate.
\end{remark}

Now that we have defined $E^{\mathrm{ce}}$-learnability, we can compare equivalence relations on indices of c.e.\ sets according to the learning power they induce.

\begin{definition}
\label{def:learn-reducibility}
Let $E^{\mathrm{ce}}$ and $F^{\mathrm{ce}}$ be equivalence relations on $\omega$ extending $=^{\mathrm{ce}}$. We say that $E^{\mathrm{ce}}$ is \define{learn-reducible} to $F^{\mathrm{ce}}$, denoted by
\[
E^{\mathrm{ce}}\leq_{\mathrm{Learn}}F^{\mathrm{ce}},
\]
if every family of nonempty c.e.\ sets that is $E^{\mathrm{ce}}$-learnable is also $F^{\mathrm{ce}}$-learnable.
\end{definition}

We write $E^{\mathrm{ce}}\equiv_{\mathrm{Learn}}F^{\mathrm{ce}}$ if both $E^{\mathrm{ce}}\leq_{\mathrm{Learn}}F^{\mathrm{ce}}$ and $F^{\mathrm{ce}}\leq_{\mathrm{Learn}}E^{\mathrm{ce}}$ hold, and we write $E^{\mathrm{ce}}\mathrel{|_{\mathrm{Learn}}}F^{\mathrm{ce}}$ if neither reduction holds.

\smallskip

Three immediate examples help fix the meaning of \Cref{def:learn-reducibility}.
\begin{enumerate}
    \item For every equivalence relation $E^{\mathrm{ce}}\, \supseteq \, =^{\mathrm{ce}}$, we have $=^{\mathrm{ce}}\, \leq_{\mathrm{Learn}}E^{\mathrm{ce}}$: any exact learner is automatically correct up to any coarser tolerance relation.
    \item $E^{\mathrm{ce}}_0$-learnability is exactly $\ex^*$-learning: eventual correctness is understood modulo finite symmetric difference.
    \item Every family of nonempty c.e.\ sets is $\Id_1$-learnable, since a learner may simply stabilise to any fixed index.
\end{enumerate}
Thus $=^{\mathrm{ce}}$ and $\Id_1$ are the natural candidates for the bottom and top of the hierarchy studied below.

Let $\mathcal R^{\mathrm{ce}}$ be the set of all equivalence relations on $\omega$ extending $=^{\mathrm{ce}}$, and let
\[
\mathcal E^{\mathrm{ce}}:=
\mathcal R^{\mathrm{ce}}/{\equiv_{\mathrm{Learn}}}.
\]
The order induced by $\leq_{\mathrm{Learn}}$ on $\mathcal E^{\mathrm{ce}}$ is denoted by $\leq$. We call $(\mathcal E^{\mathrm{ce}},\leq)$ the \define{tolerant learning hierarchy} and denote it by $\TLH$.

\section{Locking sequences and tell-tales}
\label{section:characterizations}

This section isolates the finite obstructions that control tolerant learnability.  The guiding principle is the same as in the classical theory: a successful learner must eventually reach a finite initial segment after which every compatible continuation of the target keeps the learner inside the correct semantic class.  The only change is that correctness is now measured by $E^{\ce}$ rather than by equality of c.e.\ sets.  We first formulate this as a tolerant locking sequence, then turn it into a learner-independent tell-tale theorem.  The dual notion of an absorbing set will be the main technical device in the structural results of \Cref{section:structural_properties}.

\begin{definition}
\label{def:locking_sequence}
Let $\learnerM$ be a learner, let $E^{\mathrm{ce}}$ be an equivalence relation, and let $W_e$ be a c.e.\ set. We say that a finite string $\sigma$ is an $E^{\mathrm{ce}}$-\define{locking sequence} for $\learnerM$ on $W_e$ if:
\begin{enumerate}
    \item[(a)] $\operatorname{range}(\sigma)\subseteq W_e$,
    \item[(b)] $\learnerM(\sigma)\mathrel{E^{\mathrm{ce}}}e$,
    \item[(c)] for every $\tau$, if $\operatorname{range}(\tau)\subseteq W_e$, then $\learnerM(\sigma^\smallfrown\tau)=\learnerM(\sigma)$.
\end{enumerate}
\end{definition}

As in the classical case, if a learner succeeds on a family, then for each target there must be some finite amount of positive data after which the learner is locked into the correct $E^{\mathrm{ce}}$-class. The proof is a natural adaptation of the classical Blum--Blum argument for exact identification of languages; see \cite[Theorem~3.22]{jain1999systems}.

\begin{theorem}
\label{thm:learnability_implies_lock_seq}
Let $E^{\mathrm{ce}}$ be an equivalence relation and $\K$ a family of c.e.\ sets. If a learner $\learnerM$ $E^{\mathrm{ce}}$-learns $\K$, then for every $W_e\in\K$ there is an $E^{\mathrm{ce}}$-locking sequence for $\learnerM$ on $W_e$.
\end{theorem}

\begin{proof}
Let $\learnerM$ be an $E^{\mathrm{ce}}$-learner for $\K$, let $W_e\in\K$, and suppose towards a contradiction that there is no $E^{\mathrm{ce}}$-locking sequence for $\learnerM$ on $W_e$.

Since $\learnerM$ $E^{\mathrm{ce}}$-learns $\K$, there exists some finite string $\sigma$ with $\operatorname{range}(\sigma)\subseteq W_e$ and $\learnerM(\sigma)\mathrel{E^{\mathrm{ce}}}e$. Indeed, take any enumeration $\nu$ of $W_e$; since $\learnerM$ learns $W_e$, there is some stage $s$ such that $\learnerM(\nu[s])\mathrel{E^{\mathrm{ce}}}e$.

Because we assume that there is no locking sequence, for every finite string $\sigma$ satisfying $\operatorname{range}(\sigma)\subseteq W_e$ and $\learnerM(\sigma)\mathrel{E^{\mathrm{ce}}}e$, there exists a finite string $\tau$ with $\operatorname{range}(\tau)\subseteq W_e$ such that $\learnerM(\sigma^\smallfrown\tau)\neq\learnerM(\sigma)$.

Fix an enumeration $\nu$ of $W_e$. We build an enumeration $\mu$ of $W_e$ on which $\learnerM$ fails to converge to an index $E^{\mathrm{ce}}$-equivalent to $e$.

Define finite strings $\sigma_0\sqsubseteq\sigma_1\sqsubseteq\cdots$ inductively. Let $\sigma_0$ be the empty string. Given $\sigma_s$, let $\rho_s:=\sigma_s^\smallfrown\langle\nu(s)\rangle$. If $\learnerM(\rho_s)\mathrel{E^{\mathrm{ce}}}e$, choose a finite string $\tau_s$ with $\operatorname{range}(\tau_s)\subseteq W_e$ such that $\learnerM(\rho_s^\smallfrown\tau_s)\neq\learnerM(\rho_s)$, and set $\sigma_{s+1}:=\rho_s^\smallfrown\tau_s$. Otherwise set $\sigma_{s+1}:=\rho_s$.

Let $\mu:=\bigcup_{s\in\omega}\sigma_s$. Since every value $\nu(s)$ appears in $\mu$, we have $\operatorname{range}(\mu)\supseteq\operatorname{range}(\nu)=W_e$, and by construction every entry of $\mu$ lies in $W_e$. Thus $\mu$ is an enumeration of $W_e$.

We claim that $\learnerM$ does not converge on $\mu$ to any hypothesis $k$ with $k\mathrel{E^{\mathrm{ce}}}e$. Suppose instead that there is such a $k$ and some stage $N$ such that for all $t\geq N$, $\learnerM(\mu[t])=k$. Choose $s$ so that $|\rho_s|\geq N$. Then $\rho_s$ is an initial segment of $\mu$, so $\learnerM(\rho_s)=k\mathrel{E^{\mathrm{ce}}}e$. By construction we therefore chose $\tau_s$ so that $\learnerM(\rho_s^\smallfrown\tau_s)\neq\learnerM(\rho_s)=k$. But $\rho_s^\smallfrown\tau_s=\sigma_{s+1}$ is also an initial segment of $\mu$, and its length is at least $N$, contradicting the assumed convergence to $k$.

Therefore $\learnerM$ fails to $E^{\mathrm{ce}}$-learn $W_e$, a contradiction.
\end{proof}

Locking sequences are learner-dependent. To obtain a learner-independent characterisation of $E^{\mathrm{ce}}$-learnability, we now pass to tell-tales. A tell-tale is a finite certificate attached to the target itself rather than to a particular learning strategy. The following theorem is the analogue of Angluin's classical tell-tale theorem \cite{angluin1980inductive}; the proof is the familiar one, adapted to the present setting.

\begin{theorem}
\label{theorem:telltale}
Let $\K$ be a family of c.e.\ sets and $E^{\mathrm{ce}}$ an equivalence relation. Then $\K$ is $E^{\mathrm{ce}}$-learnable if and only if for every $W_e\in\K$ there exists a \define{tell-tale}, that is, a finite set $D_{W_e}\subseteq W_e$, such that for every $W_{e'}\in\K$ with $e'\nEce{E}e$, if $D_{W_e}\subseteq W_{e'}$, then $W_{e'}\nsubseteq W_e$.
\end{theorem}

\begin{proof}
For the left-to-right direction, suppose that $\learnerM$ $E^{\mathrm{ce}}$-learns $\K$. Fix $W_e\in\K$, and let $\sigma$ be an $E^{\mathrm{ce}}$-locking sequence for $\learnerM$ on $W_e$. Set $D_{W_e}:=\operatorname{range}(\sigma)$. By \Cref{def:locking_sequence}, we have $D_{W_e}\subseteq W_e$.

Suppose towards a contradiction that there exists $W_{e'}\in\K$ such that $e'\nEce{E}e$, $D_{W_e}\subseteq W_{e'}$, and $W_{e'}\subseteq W_e$. Let $\mu$ be an enumeration of $W_{e'}$ such that $\mu[|\sigma|]=\sigma$. Since $\operatorname{range}(\mu[n])\subseteq W_e$ for every $n$, the defining property of a locking sequence implies that $\learnerM(\mu[n])=\learnerM(\sigma)$ for all $n\geq|\sigma|$; in particular, $\learnerM(\mu[n])\mathrel{E^{\mathrm{ce}}}e$ for all sufficiently large $n$. But $\mu$ is an enumeration of $W_{e'}$, and since $\learnerM$ $E^{\mathrm{ce}}$-learns $\K$, it must converge on $\mu$ to some hypothesis $k$ with $k\mathrel{E^{\mathrm{ce}}}e'$. By eventual correctness and transitivity of $E^{\mathrm{ce}}$, this yields $e\mathrel{E^{\mathrm{ce}}}e'$, a contradiction. Hence $D_{W_e}$ is a tell-tale for $W_e$.

For the right-to-left direction, assume that for every $W_e\in\K$ there is a finite set $D_{W_e}\subseteq W_e$ such that for every $W_{e'}\in\K$, if $e'\nEce{E}e$ and $D_{W_e}\subseteq W_{e'}$, then $W_{e'}\nsubseteq W_e$.

Define a learner $\learnerM$ as follows. If there is an index $e_0$ such that
\[
W_{e_0}\in\K
\quad\text{and}\quad
D_{W_{e_0}}\subseteq\operatorname{range}(\sigma)\subseteq W_{e_0},
\]
let $\learnerM(\sigma)$ be the least such $e_0$; otherwise set $\learnerM(\sigma):=0$.
We show that $\learnerM$ $E^{\mathrm{ce}}$-learns $\K$. Fix $W_i\in\K$, and let $\mu$ be an enumeration of $W_i$. Let $i_0$ be the least index such that $W_{i_0}\in\K$, $i_0\mathrel{E^{\mathrm{ce}}}i$, and $D_{W_{i_0}}\subseteq W_i\subseteq W_{i_0}$. Such an index exists, since $i$ itself has these properties.

We claim that $\learnerM$ converges to $i_0$ along $\mu$. First, since $D_{W_{i_0}}$ is finite and contained in $W_i$, there is some $n_0\in\omega$ such that
\[
D_{W_{i_0}}\subseteq\operatorname{range}(\mu[n_0])\subseteq W_i\subseteq W_{i_0}.
\]
Hence for every $n\geq n_0$, the index $i_0$ is among the candidates considered by $\learnerM(\mu[n])$, so $\learnerM(\mu[n])\leq i_0$.

Now fix $j<i_0$. We show that there is some $m_j$ such that $\learnerM(\mu[n])\neq j$ for all $n\geq m_j$.

If $W_j\notin\K$, then $j$ is never a candidate, and we may take $m_j:=0$. Suppose therefore that $W_j\in\K$. If there is no stage $n$ such that $D_{W_j}\subseteq\operatorname{range}(\mu[n])\subseteq W_j$, then again $j$ is never output and we may take $m_j:=0$. Otherwise, for some stage $n$ we have $D_{W_j}\subseteq\operatorname{range}(\mu[n])\subseteq W_j$. Since $\operatorname{range}(\mu[n])\subseteq W_i$, we get $D_{W_j}\subseteq W_i$.

If $j\nEce{E}i$, then the tell-tale property for $W_j$ implies that $W_i\nsubseteq W_j$. Hence there is some $x\in W_i\setminus W_j$. If instead $j\mathrel{E^{\mathrm{ce}}}i$, then $j$ also satisfies $j\mathrel{E^{\mathrm{ce}}}i$ and $D_{W_j}\subseteq W_i\subseteq W_j$. By minimality of $i_0$, this is impossible unless $W_i\nsubseteq W_j$. So again there is some $x\in W_i\setminus W_j$.

Choose $m_j$ such that $x\in\operatorname{range}(\mu[m_j])$. For every $n\geq m_j$, we then have $x\in\operatorname{range}(\mu[n])\setminus W_j$. Hence $j$ is no longer a candidate for $\learnerM(\mu[n])$, and therefore $\learnerM(\mu[n])\neq j$ for all $n\geq m_j$.

Let
\[
m:=\max\bigl(\{n_0\}\cup\{m_j:j<i_0\}\bigr).
\]
For every $n\geq m$, the index $i_0$ is a candidate and no $j<i_0$ is a candidate. Hence $\learnerM(\mu[n])=i_0$ for all $n\geq m$. Since $i_0\mathrel{E^{\mathrm{ce}}}i$, this is a correct $E^{\mathrm{ce}}$-hypothesis.
\end{proof}

For later structural arguments it is convenient to dualise the tell-tale viewpoint. A tell-tale is a finite fragment that blocks wrong proper subtargets. An \emph{absorbing set} is the opposite phenomenon: no matter how much of the target one has seen, one can still extend that finite information to a wrong proper subtarget. Absorbing sets will therefore serve as our main obstruction to $E^{\mathrm{ce}}$-learnability.

\begin{definition}
\label{def:absorbing}
Given a family $\K$ of c.e.\ sets and an equivalence relation $E^{\mathrm{ce}}$, we say that $W_e\in\K$ is \define{$(\K,E^{\mathrm{ce}})$-absorbing} if for every finite set $F\subseteq W_e$ there exists $W_i\in\K$ such that $i\nEce{E}e$ and $F\subseteq W_i\subsetneq W_e$. We call such a set $W_i$ a \define{$(\K,E^{\mathrm{ce}})$-absorbing witness} for $F$ on $W_e$.
\end{definition}

When $\K$ is the family of all c.e.\ sets, we simply say that $W_e$ is $E^{\mathrm{ce}}$-absorbing.
To illustrate \Cref{def:absorbing}, note that $\omega$ is $(\K,=^{\mathrm{ce}})$-absorbing for $\K:=\{\omega\restriction n:n\geq 1\}\cup\{\omega\}$.

\begin{proposition}
\label{prop:eqce_learnable_iff_no_absorbing}
Let $E^{\mathrm{ce}}$ be an equivalence relation. A family of c.e.\ sets $\K$ is $E^{\mathrm{ce}}$-learnable if and only if $\K$ has no $(\K,E^{\mathrm{ce}})$-absorbing set.
\end{proposition}

\begin{proof}
By \Cref{theorem:telltale}, $\K$ is $E^{\mathrm{ce}}$-learnable if and only if every member of $\K$ has an $E^{\mathrm{ce}}$-tell-tale.

First suppose that every set in $\K$ has an $E^{\mathrm{ce}}$-tell-tale. Fix $W_e\in\K$, and let $D_{W_e}\subseteq W_e$ be an $E^{\mathrm{ce}}$-tell-tale for $W_e$. Then there is no $W_i\in\K$ such that $i\nEce{E}e$ and $D_{W_e}\subseteq W_i\subsetneq W_e$, since this would contradict the definition of tell-tale. Hence $W_e$ is not $(\K,E^{\mathrm{ce}})$-absorbing. As $W_e$ was arbitrary, $\K$ has no $(\K,E^{\mathrm{ce}})$-absorbing set.

Conversely, suppose that $\K$ has no $(\K,E^{\mathrm{ce}})$-absorbing set. Fix $W_e\in\K$. Since $W_e$ is not $(\K,E^{\mathrm{ce}})$-absorbing, there exists a finite set $F\subseteq W_e$ such that there is no $W_i\in\K$ with $i\nEce{E}e$ and $F\subseteq W_i\subsetneq W_e$. Equivalently, for every $W_i\in\K$ with $i\nEce{E}e$, if $F\subseteq W_i$, then $W_i\nsubseteq W_e$. Thus $F$ is an $E^{\mathrm{ce}}$-tell-tale for $W_e$.
\end{proof}

Note that for $W_e$ to be $(\K,E^{\mathrm{ce}})$-absorbing it is enough to have absorbing witnesses for all initial segments $(W_e\restriction n)_{n\in\omega}$ (indeed, for any exhaustive sequence of finite subsets of $W_e$). This is because for every finite $F\subseteq W_e$ there is some $n$ such that $F\subseteq W_e\restriction n$, and any absorbing witness for $W_e\restriction n$ is also an absorbing witness for $F$. Accordingly, for a given set $W_e$ we call the sequence $(W_{u_n})_{n\in\omega}$, where $W_{u_n}$ is an absorbing witness for $W_e\restriction n$, a \define{sequence of canonical $(\K,E^{\mathrm{ce}})$-absorbing witnesses} for $W_e$. Such a sequence is of course far from unique.

The usefulness of absorbing sets is that they package nonlearnability into a single sequence that can later be manipulated combinatorially.

\begin{proposition}
\label{obs:canonical_witnesses_basic_facts}
Let $\K$ be a family of c.e.\ sets, and let $(W_{u_n})_{n\in\omega}$ be a sequence of canonical $(\K,E^{\mathrm{ce}})$-absorbing witnesses for some $W_e\in\K$. Then:
\begin{enumerate}
    \item[(a)] $(W_{u_n})_{n\in\omega}$ contains infinitely many pairwise distinct sets (that is, infinitely many pairwise $=^{\mathrm{ce}}$-inequivalent sets). In particular, ordered by inclusion, it contains either an infinite chain or an infinite antichain.
    \item[(b)] Every infinite subsequence of $(W_{u_n})_{n\in\omega}$ is itself a sequence of canonical $(\K,E^{\mathrm{ce}})$-absorbing witnesses for $W_e$.
\end{enumerate}
\end{proposition}

\begin{proof}
To prove (a), suppose towards a contradiction that the sequence contains only finitely many distinct sets, say $A_0,\dots,A_{m-1}$. Since each $A_j$ is an absorbing witness for some finite subset of $W_e$, we have $A_j\subsetneq W_e$ for each $j<m$. Choose $x_j\in W_e\setminus A_j$ for each $j<m$, and let $F:=\{x_0,\dots,x_{m-1}\}\subseteq W_e$. Since $F$ is finite, there exists $n$ such that $F\subseteq W_e\restriction n$. Now $W_{u_n}$ is an absorbing witness for $W_e\restriction n$, so $W_e\restriction n\subseteq W_{u_n}\subsetneq W_e$, and in particular $F\subseteq W_{u_n}$. But $W_{u_n}$ must equal $A_j$ for some $j<m$, contradicting $x_j\notin A_j$.

The final sentence of (a) now follows from the standard infinite chain-antichain principle for partial orders.

For (b), let $(W_{u_{n_k}})_{k\in\omega}$ be an infinite subsequence of $(W_{u_n})_{n\in\omega}$, where $n_0<n_1<n_2<\cdots$. Fix $k\in\omega$. Since $(W_{u_n})_{n\in\omega}$ is a sequence of canonical absorbing witnesses for $W_e$, the set $W_{u_{n_k}}$ is an absorbing witness for $W_e\restriction n_k$. Hence
\[
W_e\restriction n_k\subseteq W_{u_{n_k}}\subsetneq W_e\quad\text{and}\quad u_{n_k}\mathrel{\cancel{E^{\mathrm{ce}}}}e.
\]
Since $k\leq n_k$, we also have $W_e\restriction k\subseteq W_e\restriction n_k\subseteq W_{u_{n_k}}$. Therefore $W_{u_{n_k}}$ is already an absorbing witness for $W_e\restriction k$. Since this holds for every $k$, the subsequence is again a canonical witness sequence for $W_e$.
\end{proof}

\begin{theorem}
\label{more_learning_power_class_collapsed}
Let $E^{\mathrm{ce}},F^{\mathrm{ce}}$ be equivalence relations such that $E^{\mathrm{ce}}<_{\mathrm{Learn}}F^{\mathrm{ce}}$. Then there is a family $\K$ of c.e.\ sets  such that $\K$ is not $E^{\mathrm{ce}}$-learnable and all indices of sets in $\K$ belong to a single $F^{\mathrm{ce}}$-equivalence class.
\end{theorem}

\begin{proof}
Since $E^{\mathrm{ce}}<_{\mathrm{Learn}}F^{\mathrm{ce}}$, there is a family $\K$ that is $F^{\mathrm{ce}}$-learnable but not $E^{\mathrm{ce}}$-learnable. By \Cref{prop:eqce_learnable_iff_no_absorbing} there is a $(\K,E^{\mathrm{ce}})$-absorbing set $W_e\in\K$ together with a sequence $\A=(W_{u_n})_{n\in\omega}$ of canonical absorbing witnesses for $W_e$.

Suppose towards a contradiction that some infinite subsequence $\A'$ of $\A$ consists entirely of sets whose indices are not $F^{\mathrm{ce}}$-equivalent to $e$. Then $\A'$ witnesses that $W_e$ is $(\K,F^{\mathrm{ce}})$-absorbing. By \Cref{prop:eqce_learnable_iff_no_absorbing}, $\K$ would not be $F^{\mathrm{ce}}$-learnable, contradicting our choice of $\K$.

Therefore there is an infinite subsequence $\A''=(W_{u_i})_{i\in I}$ of $\A$ such that $u_i\mathrel{F^{\mathrm{ce}}}e$ for all $i\in I$. By \Cref{obs:canonical_witnesses_basic_facts}, $\A''$ is still a sequence of canonical absorbing witnesses for $W_e$. Let
\[
\K':=\{W_e\}\cup\{W_{u_i}:i\in I\}.
\]
The subsequence $\A''$ witnesses that $W_e$ is $(\K',E^{\mathrm{ce}})$-absorbing. Hence $\K'$ is not $E^{\mathrm{ce}}$-learnable, while all its members lie in a single $F^{\mathrm{ce}}$-class.
\end{proof}

This theorem will serve repeatedly as a bridge from abstract inequalities in $\TLH$ to concrete witness families.

\subsection{Universality and lowness}

We now turn from the learnability of a fixed family to the extremal points of the hierarchy. At one end are the \emph{low} relations, which do not enlarge learning power beyond exact learning; at the other end are the \emph{universal} relations, which make every family of nonempty c.e.\ sets learnable.

\begin{definition}
\label{def:low_universal_weak}
Let $E^{\mathrm{ce}}$ be an equivalence relation.
\begin{itemize}
    \item We say that $E^{\mathrm{ce}}$ is \define{low} if $E^{\mathrm{ce}}\equiv_{\mathrm{Learn}}=^{\mathrm{ce}}$.
    \item We say that $E^{\mathrm{ce}}$ is \define{universal} if every family of nonempty c.e.\ sets is $E^{\mathrm{ce}}$-learnable; equivalently, if $\{W_e:e\in\omega,\ W_e\neq\emptyset\}$ is $E^{\mathrm{ce}}$-learnable.
\end{itemize}
\end{definition}

\begin{definition}
An equivalence relation $E^{\mathrm{ce}}$ has the \define{locking property} if for every nonempty c.e.\ set $W_e$ and every sequence $(b_i)_{i\in\omega}$ of indices such that
\begin{itemize}
    \item $W_{b_i}$ is finite for every $i\in\omega$,
    \item $W_{b_i}\subseteq W_{b_{i+1}}$ for every $i\in\omega$, and
    \item $\bigcup_{i\in\omega}W_{b_i}=W_e$,
\end{itemize}
there exists a \define{locking witness}, namely some $k\in\omega$ such that $(\forall i\geq k)\,(b_i\mathrel{E^{\mathrm{ce}}}e)$.
\end{definition}

\begin{theorem}
\label{BC_characterized_by_LP}
An equivalence relation $E^{\mathrm{ce}}$ is universal if and only if $E^{\mathrm{ce}}$ has the locking property.
\end{theorem}

\begin{proof}
For the left-to-right direction, we prove the contrapositive. Suppose that $E^{\mathrm{ce}}$ does not have the locking property. Then there exist a nonempty c.e.\ set $W_e$ and a sequence $(b_i)_{i\in\omega}$ such that:
\begin{itemize}
    \item $W_{b_i}$ is finite for every $i\in\omega$,
    \item $W_{b_i}\subseteq W_{b_{i+1}}$ for every $i\in\omega$,
    \item $\bigcup_{i\in\omega}W_{b_i}=W_e$, and
    \item for every $k\in\omega$ there exists $i\geq k$ such that $b_i\nEce{E}e$.
\end{itemize}
Set $\K:=\{W_{b_i}:i\in\omega,\ W_{b_i}\neq\emptyset\}\cup\{W_e\}$. First note that $W_e$ must be infinite: if $W_e$ were finite, then for all sufficiently large $i$ we would have $W_{b_i}=W_e$, and hence $b_i=^{\mathrm{ce}}e$, contradicting the last bullet. In particular, the approximations are eventually nonempty.

We claim that $W_e$ is $(\K,E^{\mathrm{ce}})$-absorbing. Let $F\subseteq W_e$ be finite. Since the sequence $(W_{b_i})_{i\in\omega}$ is increasing and exhausts $W_e$, there exists $k$ such that $F\subseteq W_{b_k}$. By the failure of the locking property, choose $i\geq k$ large enough that $W_{b_i}\neq\emptyset$ and $b_i\nEce{E}e$. Then
\[
F\subseteq W_{b_k}\subseteq W_{b_i}\subseteq W_e.
\]
Since $W_{b_i}$ is finite and $W_e$ is infinite, we have $W_{b_i}\subsetneq W_e$. Thus $F\subseteq W_{b_i}\subsetneq W_e$ and $b_i\nEce{E}e$, proving that $W_e$ is $(\K,E^{\mathrm{ce}})$-absorbing.

By \Cref{prop:eqce_learnable_iff_no_absorbing}, $\K$ is not $E^{\mathrm{ce}}$-learnable. Hence $E^{\mathrm{ce}}$ is not universal.

For the right-to-left direction, assume that $E^{\mathrm{ce}}$ has the locking property. We show that every family of nonempty c.e.\ sets is $E^{\mathrm{ce}}$-learnable. Let $\K$ be any family of nonempty c.e.\ sets.

For each finite set $F\subseteq\omega$, fix an index $\iota(F)$ such that $W_{\iota(F)}=F$. Define a learner $\learnerM$ inductively on the length of its input by
\[
\learnerM(\varepsilon):=\iota(\emptyset),
\]
and, given a finite string $\sigma$ and an element $x$,
\[
\learnerM(\sigma^\smallfrown\langle x\rangle)=
\begin{cases}
\learnerM(\sigma),&\text{if }\iota(\operatorname{range}(\sigma^\smallfrown\langle x\rangle))\mathrel{E^{\mathrm{ce}}}\learnerM(\sigma),\\
\iota(\operatorname{range}(\sigma^\smallfrown\langle x\rangle)),&\text{otherwise.}
\end{cases}
\]
Fix an enumeration $\mu$ of some $W_e\in\K$. For each stage $s$, let $B_s:=\operatorname{range}(\mu[s])$ and $b_s:=\iota(B_s)$. Then each $B_s$ is finite, the sequence $(B_s)_{s\in\omega}$ is increasing, and $\bigcup_s B_s=W_e$. By the locking property, there exists $k\in\omega$ such that $b_s\mathrel{E^{\mathrm{ce}}}e$ for all $s\geq k$.

We now verify convergence. Suppose first that $\learnerM(\mu[k])\mathrel{E^{\mathrm{ce}}}e$. An induction on $s\geq k$ shows that
\[
\learnerM(\mu[s])=\learnerM(\mu[k]).
\]
Indeed, if the equality holds at stage $s-1$, then both $b_s$ and $\learnerM(\mu[s-1])$ are $E^{\mathrm{ce}}$-equivalent to $e$, so the definition of $\learnerM$ leaves the hypothesis unchanged at stage $s$.

Suppose instead that $\learnerM(\mu[k])\nEce{E}e$. Since $b_{k+1}\mathrel{E^{\mathrm{ce}}}e$, we have $b_{k+1}\nEce{E}\learnerM(\mu[k])$, and hence
\[
\learnerM(\mu[k+1])=b_{k+1}\mathrel{E^{\mathrm{ce}}}e.
\]
The same induction, now starting at stage $k+1$, shows that the output remains constant thereafter.

Thus $\learnerM$ converges to a hypothesis in the $E^{\mathrm{ce}}$-class of $e$. Since $\mu$ and $W_e$ were arbitrary, $E^{\mathrm{ce}}$ is universal.
\end{proof}

So universality is governed by a very simple eventual-stability condition on increasing finite approximations. The next corollaries show that this condition is already quite restrictive.

\begin{corollary}
Let $E^{\mathrm{ce}}$ be an equivalence relation all of whose equivalence classes contain only finitely many pairwise $=^{\mathrm{ce}}$-inequivalent elements. Then $E^{\mathrm{ce}}$ is not universal.
\end{corollary}

\begin{proof}
Let $e$ be an index for $\omega$, and for each $n\in\omega$ let $b_n$ be an index such that $W_{b_n}=\{0,\dots,n-1\}$. Then each $W_{b_n}$ is finite, the sequence is increasing, and $\bigcup_n W_{b_n}=\omega=W_e$. If $E^{\mathrm{ce}}$ were universal, then by \Cref{BC_characterized_by_LP} there would be some $k$ such that $b_n\mathrel{E^{\mathrm{ce}}}e$ for all $n\geq k$. But the sets $W_{b_n}$ are pairwise $=^{\mathrm{ce}}$-inequivalent, whereas by assumption the $E^{\mathrm{ce}}$-class of $e$ contains only finitely many such sets. This is impossible.
\end{proof}

\begin{corollary}
Let $E^{\mathrm{ce}}$ be an equivalence relation. If there is an infinite c.e.\ set that is not $E^{\mathrm{ce}}$-equivalent to any of its proper c.e.\ subsets, then $E^{\mathrm{ce}}$ is not universal.
\end{corollary}

\begin{proof}
Let $W_e$ be an infinite c.e.\ set such that no proper c.e.\ subset of $W_e$ is $E^{\mathrm{ce}}$-equivalent to $W_e$. Fix a computable increasing sequence $(b_n)_{n\in\omega}$ of indices such that each $W_{b_n}$ is finite and $\bigcup_n W_{b_n}=W_e$. Since each $W_{b_n}$ is a proper subset of $W_e$, we have $b_n\nEce{E}e$ for every $n$. Thus $(b_n)_{n\in\omega}$ has no locking witness, so $E^{\mathrm{ce}}$ fails the locking property. By \Cref{BC_characterized_by_LP}, $E^{\mathrm{ce}}$ is not universal.
\end{proof}

\section{Structural properties of the tolerant learning hierarchy}
\label{section:structural_properties}

We now pass from individual tolerance relations to the global order $\TLH$.  The results in this section show that the hierarchy has a well-behaved meet structure but is highly nonlattice-like above.  We first identify the extremal elements and construct large chains and antichains.  We then prove that meets are represented by intersections of equivalence relations.  Finally, we show that joins can fail dramatically and embed copies of $\mathcal P_{\mathrm{fin}}(2^\omega)$ below every non-bottom degree and above every non-top degree, yielding atomlessness and coatomlessness.

\smallskip

The following easy observation says that collapsing equivalence classes can only increase learning power.

\begin{lemma}
\label{lem:quotienting_increases_learning_power}
Let $E^{\mathrm{ce}}$ and $F^{\mathrm{ce}}$ be equivalence relations. If $E^{\mathrm{ce}}\subseteq F^{\mathrm{ce}}$, then $E^{\mathrm{ce}}\leq_{\mathrm{Learn}}F^{\mathrm{ce}}$.
\end{lemma}

\begin{proof}
Suppose that $\K$ is $E^{\mathrm{ce}}$-learnable via a learner $\learnerM$. Since $E^{\mathrm{ce}}\subseteq F^{\mathrm{ce}}$, the same learner is automatically correct as an $F^{\mathrm{ce}}$-learner for $\K$.
\end{proof}

Note that the converse implication fails in general. Also, the above implication in strict form fails. For example, let $E^{\mathrm{ce}}$ be the equivalence relation generated by $=^{\mathrm{ce}}\, \cup \, \{(e,i)\}$ for some fixed $e\nEce{=}i$. Then $=^{\mathrm{ce}}\subsetneq E^{\mathrm{ce}}$, but $=^{\mathrm{ce}}\equiv_{\mathrm{Learn}} E^{\mathrm{ce}}$: otherwise, by \Cref{more_learning_power_class_collapsed}, some single $E^{\mathrm{ce}}$-class would fail to be $=^{\mathrm{ce}}$-learnable, which is clearly impossible.

\begin{proposition}
\label{least_and_greatest}
The tolerant learning hierarchy $\TLH$ is bounded. Its bottom and top elements are respectively $[=^{\mathrm{ce}}]_{\equiv_{\mathrm{Learn}}}$ and $[\Id_1]_{\equiv_{\mathrm{Learn}}}$. In particular, $=^{\mathrm{ce}}$ is low and $\Id_1$ is universal.
\end{proposition}

\begin{proof}
Every equivalence relation considered in this paper extends $=^{\mathrm{ce}}$, so by \Cref{lem:quotienting_increases_learning_power} we have $=^{\mathrm{ce}}\leq_{\mathrm{Learn}}E^{\mathrm{ce}}$ for every $E^{\mathrm{ce}}$. On the other hand, $\Id_1$ learns every family by definition.
\end{proof}

Thus $\TLH$ has the simplest possible extremal elements. The rest of the section shows that everything in between is highly nontrivial.

\subsection{Height and width of the tolerant learning hierarchy}

A first indication that $\TLH$ is far from linear is that it already contains many incompatible ways of enlarging learning power. We now show that this happens on the scale of the continuum.

\begin{itemize}
    \item Let $P_k:=\{p_k^n:n\geq 1\}$, where $p_k$ is the $k$-th prime.
    \item For each $m\geq 1$, let $P_k[m]:=\{p_k^n:1\leq n\leq m\}$, and let $\mathcal P_k:=\{P_k[m]:m\geq 1\}\cup\{P_k\}$. Ordered by inclusion, $\mathcal P_k$ is a chain of order type $\omega+1$. The families $\mathcal P_k$ are pairwise disjoint.
    \item Given $C\subseteq\omega$, define $E(C)^{\mathrm{ce}}$ by
    \[
    j\mathrel{E(C)^{\mathrm{ce}}}\ell
    \iff
    j=^{\mathrm{ce}}\ell\ \lor\ \exists n\in C\,\bigl(W_j\in\mathcal P_n\text{ and }W_\ell\in\mathcal P_n\bigr).
    \]
\end{itemize}

\begin{proposition}
\label{prop:CsubseteqDiffEClernED}
Let $C,D\subseteq\omega$. Then
\[
C\subseteq D\quad\Longleftrightarrow\quad E(C)^{\mathrm{ce}}\leq_{\mathrm{Learn}}E(D)^{\mathrm{ce}}.
\]
\end{proposition}

\begin{proof}
If $C\subseteq D$, then by definition $E(C)^{\mathrm{ce}}\subseteq E(D)^{\mathrm{ce}}$, so the result follows from \Cref{lem:quotienting_increases_learning_power}.

Conversely, suppose that $C\not\subseteq D$, and choose $n\in C\setminus D$. The family $\mathcal P_n$ is trivially $E(C)^{\mathrm{ce}}$-learnable, since all its members lie in one $E(C)^{\mathrm{ce}}$-class. On the other hand, $P_n$ is $(\mathcal P_n,E(D)^{\mathrm{ce}})$-absorbing: given any finite $F\subseteq P_n$, choose $m\geq 1$ with $F\subseteq P_n[m]\subsetneq P_n$; since $n\notin D$, we have $P_n[m]\nEce{E(D)}P_n$. By \Cref{prop:eqce_learnable_iff_no_absorbing}, $\mathcal P_n$ is not $E(D)^{\mathrm{ce}}$-learnable.
\end{proof}

\begin{corollary}
\label{uncountable_chains_antichains}
The powerset of the natural numbers under inclusion, $(\mathcal P(\omega),\subseteq)$, embeds into $\TLH$. In particular, $\TLH$ has antichains and chains of size $2^{\aleph_0}$, and both its width and its height are exactly $2^{\aleph_0}$.
\end{corollary}

\begin{proof}
Define $\iota(A):=[E(A)^{\mathrm{ce}}]_{\equiv_{\mathrm{Learn}}}$ for $A\subseteq\omega$. By \Cref{prop:CsubseteqDiffEClernED}, this is an order embedding of $(\mathcal P(\omega),\subseteq)$ into $\TLH$. Since there are only $2^{\aleph_0}$ equivalence relations on $\omega$, the upper bounds on the width and height follow as well.
\end{proof}

\subsection{The tolerant learning hierarchy is a lower semilattice}

We now turn to the algebraic structure of $\TLH$. The next proposition shows that meets are given by the most obvious candidate, namely intersection of equivalence relations.

\begin{proposition}
\label{proposition:meetintersection}
Let $E^{\mathrm{ce}}$ and $F^{\mathrm{ce}}$ be equivalence relations. A family $\K$ is $(E^{\mathrm{ce}}\cap F^{\mathrm{ce}})$-learnable if and only if it is both $E^{\mathrm{ce}}$-learnable and $F^{\mathrm{ce}}$-learnable.
\end{proposition}

\begin{proof}
The left-to-right direction is immediate from \Cref{lem:quotienting_increases_learning_power}, since $E^{\mathrm{ce}}\cap F^{\mathrm{ce}}\subseteq E^{\mathrm{ce}}$ and $E^{\mathrm{ce}}\cap F^{\mathrm{ce}}\subseteq F^{\mathrm{ce}}$.

For the converse, we prove the contrapositive. Assume that $\K$ is not $(E^{\mathrm{ce}}\cap F^{\mathrm{ce}})$-learnable. By \Cref{prop:eqce_learnable_iff_no_absorbing}, there is a $(\K,E^{\mathrm{ce}}\cap F^{\mathrm{ce}})$-absorbing set $W_e\in\K$ together with a sequence $(W_{u_i})_{i\in\omega}$ of canonical absorbing witnesses for $W_e$.

For each $i$, since $u_i\mathrel{\cancel{(E^{\mathrm{ce}}\cap F^{\mathrm{ce}})}}e$, at least one of the relations $u_i\nEce{E}e$ or $u_i\nEce{F}e$ holds. One of these alternatives must hold for infinitely many $i$. Without loss of generality, suppose that $u_i\nEce{E}e$ for infinitely many $i$, and pass to the corresponding infinite subsequence. By \Cref{obs:canonical_witnesses_basic_facts}, this subsequence is still canonical for $W_e$. Hence it witnesses that $W_e$ is $(\K,E^{\mathrm{ce}})$-absorbing. By \Cref{prop:eqce_learnable_iff_no_absorbing}, the family $\K$ is not $E^{\mathrm{ce}}$-learnable.

Thus if $\K$ fails to be $(E^{\mathrm{ce}}\cap F^{\mathrm{ce}})$-learnable, then it fails to be $E^{\mathrm{ce}}$-learnable or it fails to be $F^{\mathrm{ce}}$-learnable. This proves the contrapositive.
\end{proof}

\begin{corollary}
\label{lower_semilattice}
The structure $\TLH$ is a lower semilattice. For every pair of equivalence relations $E^{\mathrm{ce}}$ and $F^{\mathrm{ce}}$, the meet of $[E^{\mathrm{ce}}]_{\equiv_{\mathrm{Learn}}}$ and $[F^{\mathrm{ce}}]_{\equiv_{\mathrm{Learn}}}$ is $[E^{\mathrm{ce}}\cap F^{\mathrm{ce}}]_{\equiv_{\mathrm{Learn}}}$.
\end{corollary}

Thus intersection gives the right notion of meet. The natural next question is whether some corresponding construction yields joins. The answer is negative. In particular, the equivalence closure of the union $E^{\mathrm{ce}}\cup F^{\mathrm{ce}}$ need not be a join. Indeed, let $E^{\mathrm{ce}}$ be obtained from $=^{\mathrm{ce}}$ by collapsing all indices of $\omega$ and $\{0\}$ into one class, and let $F^{\mathrm{ce}}$ be obtained from $=^{\mathrm{ce}}$ by collapsing all indices of $\omega\restriction n$, for $n\geq 1$, into one class. It is easy to check that $E^{\mathrm{ce}}\equiv_{\mathrm{Learn}}F^{\mathrm{ce}}\equiv_{\mathrm{Learn}}=^{\mathrm{ce}}$. Let $G^{\mathrm{ce}}$ be the smallest equivalence relation containing $E^{\mathrm{ce}}\cup F^{\mathrm{ce}}$. Then all sets $\omega,\omega\restriction 1,\omega\restriction 2,\omega\restriction 3,\dots$ belong to one $G^{\mathrm{ce}}$-class, so $\{\omega\restriction n:n\geq 1\}\cup\{\omega\}$ is $G^{\mathrm{ce}}$-learnable. But this family is not $=^{\mathrm{ce}}$-learnable. Hence $=^{\mathrm{ce}}<_{\mathrm{Learn}}G^{\mathrm{ce}}$, so $G^{\mathrm{ce}}$ cannot be the join of $[E^{\mathrm{ce}}]_{\equiv_{\mathrm{Learn}}}$ and $[F^{\mathrm{ce}}]_{\equiv_{\mathrm{Learn}}}$.

We now show that joins fail in a much stronger sense: some pairs have no minimal common upper bound at all.

\begin{definition}
Let $E^{\mathrm{ce}}$ and $F^{\mathrm{ce}}$ be equivalence relations. We say that $E^{\mathrm{ce}}$ is an \define{almost quotient} of $F^{\mathrm{ce}}$ if for every index $e\in\omega$, the set
\[
\{W_i:i\mathrel{F^{\mathrm{ce}}}e\text{ and }i\mathrel{\cancel{E^{\mathrm{ce}}}}e\}
\]
is finite.
\end{definition}

\begin{lemma}
\label{almost_quotient_increases_learning}
If $E^{\mathrm{ce}}$ is an almost quotient of $F^{\mathrm{ce}}$, then $F^{\mathrm{ce}}\leq_{\mathrm{Learn}}E^{\mathrm{ce}}$.
\end{lemma}

\begin{proof}
Let $E^{\mathrm{ce}}$ be an almost quotient of $F^{\mathrm{ce}}$. We show that every family that is not $E^{\mathrm{ce}}$-learnable is also not $F^{\mathrm{ce}}$-learnable.

Let $\K$ be a family of c.e.\ sets that is not $E^{\mathrm{ce}}$-learnable. By \Cref{prop:eqce_learnable_iff_no_absorbing}, there is a $(\K,E^{\mathrm{ce}})$-absorbing set $W_e\in\K$ together with a canonical sequence $(W_{u_n})_{n\in\omega}$ of $(\K,E^{\mathrm{ce}})$-absorbing witnesses for $W_e$. In particular, $u_n\mathrel{\cancel{E^{\mathrm{ce}}}}e$ for every $n$.

By \Cref{obs:canonical_witnesses_basic_facts}(a), the sequence $(W_{u_n})_{n\in\omega}$ contains infinitely many pairwise distinct sets. Passing to an infinite subsequence if necessary, we may assume that the sets $W_{u_{n_k}}$ are pairwise $=^{\mathrm{ce}}$-inequivalent. By \Cref{obs:canonical_witnesses_basic_facts}(b), this subsequence is still a canonical witness sequence for $W_e$.

Now suppose that $u_{n_k}\mathrel{F^{\mathrm{ce}}}e$. Since also $u_{n_k}\mathrel{\cancel{E^{\mathrm{ce}}}}e$, the corresponding set $W_{u_{n_k}}$ belongs to
\[
\{W_i:i\mathrel{F^{\mathrm{ce}}}e\text{ and }i\mathrel{\cancel{E^{\mathrm{ce}}}}e\}.
\]
Because $E^{\mathrm{ce}}$ is an almost quotient of $F^{\mathrm{ce}}$, this set is finite. But the sets $W_{u_{n_k}}$ are pairwise distinct, so only finitely many $k$ can satisfy $u_{n_k}\mathrel{F^{\mathrm{ce}}}e$.

Therefore there is an infinite subsequence $(W_{u_{n_k}})_{k\in I}$ such that $u_{n_k}\mathrel{\cancel{F^{\mathrm{ce}}}}e$ for every $k\in I$. By \Cref{obs:canonical_witnesses_basic_facts}(b), this is again a canonical witness sequence for $W_e$. Thus $W_e$ is $(\K,F^{\mathrm{ce}})$-absorbing, and by \Cref{prop:eqce_learnable_iff_no_absorbing}, $\K$ is not $F^{\mathrm{ce}}$-learnable.
\end{proof}

\begin{theorem}
\label{thm:not_upper_semilattice}
The tolerant learning hierarchy $\TLH$ is not an upper semilattice. In fact, there are equivalence relations $E^{\mathrm{ce}}$ and $F^{\mathrm{ce}}$ such that
\[
\bigl\{[D^{\mathrm{ce}}]_{\equiv_{\mathrm{Learn}}}:E^{\mathrm{ce}},F^{\mathrm{ce}}\leq_{\mathrm{Learn}}D^{\mathrm{ce}}\bigr\}
\]
has no minimal element.
\end{theorem}

\begin{proof}
Using the fixed computable pairing function, identify subsets of $\omega^2$
with subsets of $\omega$ throughout the proof. We define two simple equivalence relations whose common upper bounds can be characterised by a pair of eventual-collapse conditions.

For $i,j\in\omega$, let
\begin{align*}
X_{i,j}&:=\{(i',j')\in\omega^2:i'\leq i\text{ and }j'\leq j\},\\
X_{i,\omega}&:=\{(i',j')\in\omega^2:i'\leq i\text{ and }j'\in\omega\},\\
X_{\omega,\omega}&:=\omega^2.
\end{align*}
For each $i\in\omega$, let
\[
\mathcal E_i:=\{X_{i,j}:j\in\omega\}\cup\{X_{i,\omega}\},
\]
and let
\[
\mathcal F:=\{X_{i,\omega}:i\in\omega\}\cup\{X_{\omega,\omega}\}.
\]
Define equivalence relations $E^{\mathrm{ce}}$ and $F^{\mathrm{ce}}$ by
\[
u\mathrel{E^{\mathrm{ce}}}v\iff u=^{\mathrm{ce}}v\ \lor\ \exists i\in\omega\,(W_u,W_v\in\mathcal E_i),
\]
and
\[
u\mathrel{F^{\mathrm{ce}}}v\iff u=^{\mathrm{ce}}v\ \lor\ (W_u,W_v\in\mathcal F).
\]
Thus the nontrivial $E^{\mathrm{ce}}$-classes are exactly the columns $\mathcal E_i$, while the only nontrivial $F^{\mathrm{ce}}$-class is $\mathcal F$; see \Cref{fig:equiv_rels_2d_grid}.

\begin{figure}[H]
    \begin{center}
        \begin{tikzpicture}[every node/.style={inner sep=0.8pt},labelstyle/.style={font=\large\bfseries}]
\begin{scope}[xshift=0cm]
\node[labelstyle] at (1.2,3.9) {$E^{\mathrm{ce}}$};
\foreach \x in {0,1.2,2.4} {
  \draw[rounded corners] (\x-0.45,-0.25) rectangle (\x+0.45,3.25);
}
\node at (0,0) {$X_{0,0}$};
\node at (1.2,0) {$X_{1,0}$};
\node at (2.4,0) {$X_{2,0}$};
\node at (0,0.85) {$X_{0,1}$};
\node at (1.2,0.85) {$X_{1,1}$};
\node at (2.4,0.85) {$X_{2,1}$};
\node at (0,1.7) {$X_{0,2}$};
\node at (1.2,1.7) {$X_{1,2}$};
\node at (2.4,1.7) {$X_{2,2}$};
\node at (0,2.35) {$\vdots$};
\node at (1.2,2.35) {$\vdots$};
\node at (2.4,2.35) {$\vdots$};
\node at (0,2.9) {$X_{0,\omega}$};
\node at (1.2,2.9) {$X_{1,\omega}$};
\node at (2.4,2.9) {$X_{2,\omega}$};
\node at (3.4,2.9) {$\cdots$};
\node at (4.3,2.9) {$X_{\omega,\omega}$};
\end{scope}
\begin{scope}[xshift=6cm]
\node[labelstyle] at (1.2,3.9) {$F^{\mathrm{ce}}$};
\draw[rounded corners] (-0.45,2.6) rectangle (4.7,3.25);
\node at (0,0) {$X_{0,0}$};
\node at (1.2,0) {$X_{1,0}$};
\node at (2.4,0) {$X_{2,0}$};
\node at (0,0.85) {$X_{0,1}$};
\node at (1.2,0.85) {$X_{1,1}$};
\node at (2.4,0.85) {$X_{2,1}$};
\node at (0,1.7) {$X_{0,2}$};
\node at (1.2,1.7) {$X_{1,2}$};
\node at (2.4,1.7) {$X_{2,2}$};
\node at (0,2.35) {$\vdots$};
\node at (1.2,2.35) {$\vdots$};
\node at (2.4,2.35) {$\vdots$};
\node at (0,2.9) {$X_{0,\omega}$};
\node at (1.2,2.9) {$X_{1,\omega}$};
\node at (2.4,2.9) {$X_{2,\omega}$};
\node at (3.4,2.9) {$\cdots$};
\node at (4.3,2.9) {$X_{\omega,\omega}$};
\end{scope}
\end{tikzpicture}
    \end{center}
    \caption{The equivalence relations $E^{\mathrm{ce}}$ and $F^{\mathrm{ce}}$ used in the proof of \Cref{thm:not_upper_semilattice}. Rectangular regions indicate nontrivial equivalence classes; all other sets are singletons modulo $=^{\mathrm{ce}}$.}
    \label{fig:equiv_rels_2d_grid}
\end{figure}
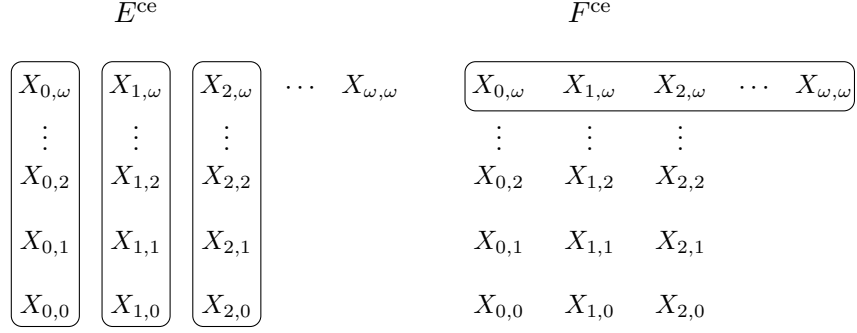

We claim that for an equivalence relation $D^{\mathrm{ce}}$ the following are equivalent:
\begin{itemize}
    \item[(i)] $E^{\mathrm{ce}},F^{\mathrm{ce}}\leq_{\mathrm{Learn}}D^{\mathrm{ce}}$;
    \item[(ii)] both of the following hold:
    \begin{itemize}
        \item[(a)] for every $i\in\omega$, for all but finitely many $j\in\omega$, we have $X_{i,j}\mathrel{D^{\mathrm{ce}}}X_{i,\omega}$;
        \item[(b)] for all but finitely many $i\in\omega$, we have $X_{i,\omega}\mathrel{D^{\mathrm{ce}}}X_{\omega,\omega}$.
    \end{itemize}
\end{itemize}

\medskip
\noindent\emph{Proof of $(i)\Rightarrow(ii)$.}
Assume that $E^{\mathrm{ce}},F^{\mathrm{ce}}\leq_{\mathrm{Learn}}D^{\mathrm{ce}}$.

Suppose that (a) fails. Then for some fixed $i$ there are infinitely many $j$ such that $X_{i,j}\mathrel{\cancel{D^{\mathrm{ce}}}}X_{i,\omega}$. Let
\[
\K_i:=\{X_{i,j}:j\in\omega\}\cup\{X_{i,\omega}\}.
\]
Since all members of $\K_i$ lie in one $E^{\mathrm{ce}}$-equivalence class, the family $\K_i$ is $E^{\mathrm{ce}}$-learnable. On the other hand, $X_{i,\omega}$ is $(\K_i,D^{\mathrm{ce}})$-absorbing: given a finite $F\subseteq X_{i,\omega}$, choose $j$ so large that $F\subseteq X_{i,j}$ and still $X_{i,j}\mathrel{\cancel{D^{\mathrm{ce}}}}X_{i,\omega}$. Then $F\subseteq X_{i,j}\subsetneq X_{i,\omega}$, so by \Cref{prop:eqce_learnable_iff_no_absorbing} the family $\K_i$ is not $D^{\mathrm{ce}}$-learnable. This contradicts $E^{\mathrm{ce}}\leq_{\mathrm{Learn}}D^{\mathrm{ce}}$.

Suppose now that (b) fails. Then there are infinitely many $i$ such that $X_{i,\omega}\mathrel{\cancel{D^{\mathrm{ce}}}}X_{\omega,\omega}$. Let
\[
\mathcal L:=\{X_{i,\omega}:i\in\omega\}\cup\{X_{\omega,\omega}\}.
\]
Since all members of $\mathcal L$ lie in one $F^{\mathrm{ce}}$-equivalence class, the family $\mathcal L$ is $F^{\mathrm{ce}}$-learnable. But $X_{\omega,\omega}$ is $(\mathcal L,D^{\mathrm{ce}})$-absorbing: if $F\subseteq X_{\omega,\omega}$ is finite, then for all sufficiently large $i$ we have $F\subseteq X_{i,\omega}$, and among those $i$ there are infinitely many with $X_{i,\omega}\mathrel{\cancel{D^{\mathrm{ce}}}}X_{\omega,\omega}$. Again \Cref{prop:eqce_learnable_iff_no_absorbing} gives a contradiction. Hence both (a) and (b) hold.

\medskip
\noindent\emph{Proof of $(ii)\Rightarrow(i)$.}
Assume that $D^{\mathrm{ce}}$ satisfies (a) and (b). 

First, we show that $D^{ce}$ is an almost quotient of $E^{ce}$. Let $W_e$ be any c.e.\ set. If $W_e$ is not in any $\mathcal{E}_j$, $j\in \omega$, then $W_e$ is a singleton equivalence class (modulo $=^{ce}$) in $E^{ce}$ and therefore $\{W_i : i E^{ce} e \text{ and } i \mathrel{\cancel{D^{ce}}} e\} = \emptyset$. If there is a $j \in \omega$ such that $W_e \in \mathcal{E}_j$ then by (a) $\{W_i : i E^{ce} e \text{ and } i \mathrel{\cancel{ D^{ce}}} e\}$ is finite. So, $D^{ce}$ is an almost quotient of $E^{ce}$.

Second, we show that $D^{ce}$ is an almost quotient of $F^{ce}$. Let $W_e$ be any c.e.\ set. If $W_e$ is not in $\mathcal{F}$, then $W_e$ is a singleton equivalence class (modulo $=^{ce}$) in $F^{ce}$ and therefore $\{W_i : i F^{ce} e \text{ and } i \mathrel{\cancel{ D^{ce}}} e\} = \emptyset$. If $W_e$ is in $\mathcal{F}$, then by (b) $\{W_i : i F^{ce} e \text{ and } i \mathrel{\cancel{ D^{ce}}} e\}$ is finite. So, $D^{ce}$ is an almost quotient of $F^{ce}$.

Therefore, by \Cref{almost_quotient_increases_learning}, $E^{ce},F^{ce}\learnreducible{}D^{ce}$. This proves the claim.

\medskip
Now let $D^{\mathrm{ce}}$ be any common upper bound of $E^{\mathrm{ce}}$ and $F^{\mathrm{ce}}$. By the claim, $D^{\mathrm{ce}}$ satisfies (a) and (b). Choose $k\in\omega$ such that for every $i\geq k$ we have $X_{i,\omega}\mathrel{D^{\mathrm{ce}}}X_{\omega,\omega}$. Using (a), choose inductively a strictly increasing sequence
\[
l_k<l_{k+1}<l_{k+2}<\cdots
\]
such that for every $i\geq k$ and every $j\geq l_i$, we have $X_{i,j}\mathrel{D^{\mathrm{ce}}}X_{i,\omega}$. By transitivity, $X_{i,l_i}\mathrel{D^{\mathrm{ce}}}X_{\omega,\omega}$ for all $i\geq k$. Moreover,
\[
X_{i,l_i}\subsetneq X_{i+1,l_{i+1}}\qquad(i\geq k),
\]
each $X_{i,l_i}$ is finite, and $\bigcup_{i\geq k}X_{i,l_i}=X_{\omega,\omega}$.

Let $S:=\{X_{i,l_i}:i\geq k\}$. Define an equivalence relation $\widehat D^{\mathrm{ce}}$ by splitting each set in $S$ off from its old $D^{\mathrm{ce}}$-class:
\[
u\mathrel{\widehat D^{\mathrm{ce}}}v
\iff
\bigl(W_u=W_v=X_{i,l_i}\text{ for some }i\geq k\bigr)
\ \lor\ 
\bigl(u\mathrel{D^{\mathrm{ce}}}v\text{ and }W_u,W_v\notin S\bigr).
\]
This is an equivalence relation, and by construction $\widehat D^{\mathrm{ce}}\subseteq D^{\mathrm{ce}}$. Therefore, by \Cref{lem:quotienting_increases_learning_power},
\[
\widehat D^{\mathrm{ce}}\leq_{\mathrm{Learn}}D^{\mathrm{ce}}.
\]

We claim that $\widehat D^{\mathrm{ce}}$ still satisfies (a) and (b). For each fixed $i\geq k$, we removed only the single set $X_{i,l_i}$ from the $D^{\mathrm{ce}}$-class of $X_{i,\omega}$, so still all but finitely many $j$ satisfy $X_{i,j}\mathrel{\widehat D^{\mathrm{ce}}}X_{i,\omega}$. For $i<k$, condition (a) is unchanged. Condition (b) is unchanged as well, since neither any $X_{i,\omega}$ nor $X_{\omega,\omega}$ was split off. Hence, by the claim, $E^{\mathrm{ce}},F^{\mathrm{ce}}\leq_{\mathrm{Learn}}\widehat D^{\mathrm{ce}}$.

Finally, consider the family
\[
\mathcal K:=\{X_{i,l_i}:i\geq k\}\cup\{X_{\omega,\omega}\}.
\]
Since every $X_{i,l_i}$ is $D^{\mathrm{ce}}$-equivalent to $X_{\omega,\omega}$, the family $\mathcal K$ is $D^{\mathrm{ce}}$-learnable. But $X_{\omega,\omega}$ is $(\mathcal K,\widehat D^{\mathrm{ce}})$-absorbing: if $F\subseteq X_{\omega,\omega}$ is finite, then for some $i\geq k$ we have $F\subseteq X_{i,l_i}\subsetneq X_{\omega,\omega}$, while by construction $X_{i,l_i}\mathrel{\cancel{\widehat D^{\mathrm{ce}}}}X_{\omega,\omega}$. Hence $\mathcal K$ is not $\widehat D^{\mathrm{ce}}$-learnable by \Cref{prop:eqce_learnable_iff_no_absorbing}.

Therefore $\widehat D^{\mathrm{ce}}<_{\mathrm{Learn}}D^{\mathrm{ce}}$. Since $D^{\mathrm{ce}}$ was an arbitrary common upper bound of $E^{\mathrm{ce}}$ and $F^{\mathrm{ce}}$, the set of common upper bounds has no minimal element.
\end{proof}

Thus the failure of joins is not a marginal pathology. Even among common upper bounds there need not be any minimal one.

\subsection{The tolerant learning hierarchy is atomless and coatomless}

To prove atomlessness and coatomlessness, we establish a stronger local statement: one can embed the following order below every non-bottom degree and above every non-top degree:
\[
(\mathcal P_{\mathrm{fin}}(2^\omega),\subseteq).
\]

\begin{lemma}
\label{lem:downward_embed_pfin_cantor}
Let $F^{\mathrm{ce}}$ be an equivalence relation such that
$=^{\mathrm{ce}}<_{\mathrm{Learn}}F^{\mathrm{ce}}$, and let $W_e$
be a c.e.\ set together with a sequence $(W_{u_i})_{i\in\omega}$ of canonical $=^{\mathrm{ce}}$-absorbing witnesses for $W_e$ such that:
\begin{enumerate}
    \item[(a)] for all $i\neq j$, $u_i\nEce{=}u_j$,
    \item[(b)] for every $i$, $u_i\mathrel{F^{\mathrm{ce}}}e$.
\end{enumerate}
Then there is an embedding $\iota$ from $(\mathcal P_{\mathrm{fin}}(\Cantor),\subseteq)$ to $(\mathcal E^{\mathrm{ce}},\leq)$ such that:
\begin{itemize}
    \item $\iota(\emptyset)=[=^{\mathrm{ce}}]_{\equiv_{\mathrm{Learn}}}$,
    \item for every $P\in\mathcal P_{\mathrm{fin}}(\Cantor)$, we have $\iota(P)<[F^{\mathrm{ce}}]_{\equiv_{\mathrm{Learn}}}$.
\end{itemize}
\end{lemma}

\begin{proof}
Since $W_e$ is $=^{\mathrm{ce}}$-absorbing, it is infinite. After discarding finitely many initial terms and reindexing the witness sequence, we may therefore assume that every $W_{u_i}$ is nonempty; by \Cref{obs:canonical_witnesses_basic_facts}, the resulting sequence is still canonical.

Let $T:=2^{<\omega}$ be the full binary tree. Choose an injective map
$r\colon T\to\omega$ such that
\[
\sigma\sqsubset\tau \quad\Longrightarrow\quad r(\sigma)<r(\tau),
\]
and set
\[
\ell(\sigma):=u_{r(\sigma)}.
\]
For every path $\pi\in[T]$, the sequence
$(W_{\ell(\pi\restriction n)})_{n\in\omega}$ is an infinite
subsequence of the original witness sequence and hence, by
\Cref{obs:canonical_witnesses_basic_facts}, is again a sequence of canonical
$=^{\mathrm{ce}}$-absorbing witnesses for $W_e$.

For a finite set $P\subseteq[T]$, let $\K_P$ be the family consisting of $W_e$ together with all sets of the form $W_{\ell(\pi\restriction n)}$ where $\pi\in P$ and $n\in\omega$. Define an equivalence relation $E^{\mathrm{ce}}_P$ by
\[
j\mathrel{E^{\mathrm{ce}}_P}k
\iff
j=^{\mathrm{ce}}k\ \lor\ \bigl(W_j,W_k\in\K_P\bigr).
\]
So $E^{\mathrm{ce}}_P$ is obtained from $=^{\mathrm{ce}}$ by collapsing all indices of members of $\K_P$ into one equivalence class.

We claim that for all finite $P_1,P_2\subseteq[T]$,
\[
P_1\subseteq P_2\quad\Longleftrightarrow\quad E^{\mathrm{ce}}_{P_1}\leq_{\mathrm{Learn}}E^{\mathrm{ce}}_{P_2}.
\]
The forward implication is immediate from inclusion and \Cref{lem:quotienting_increases_learning_power}.

For the converse, suppose $P_1\not\subseteq P_2$, and choose $\pi\in P_1\setminus P_2$. Since $P_2$ is finite and $\pi$ differs from every path in $P_2$, there is some $N$ such that for every $n\geq N$ and every $\rho\in P_2$, the string $\pi\restriction n$ is not an initial segment of $\rho$. Consequently, for every $n\geq N$, the set $W_{\ell(\pi\restriction n)}$ belongs to $\K_{P_1}$ but not to $\K_{P_2}$.

The tail $(W_{\ell(\pi\restriction n)})_{n\geq N}$ is again a canonical witness sequence for $W_e$. Moreover, for each $n\geq N$ we have $W_{\ell(\pi\restriction n)}\mathrel{\cancel{E^{\mathrm{ce}}_{P_2}}}W_e$, since $E^{\mathrm{ce}}_{P_2}$ only collapses members of $\K_{P_2}$ with $W_e$, and by construction $W_{\ell(\pi\restriction n)}\neq W_e$. Therefore the tail witnesses that $W_e$ is $(\K_{P_1},E^{\mathrm{ce}}_{P_2})$-absorbing. By \Cref{prop:eqce_learnable_iff_no_absorbing}, the family $\K_{P_1}$ is not $E^{\mathrm{ce}}_{P_2}$-learnable.

On the other hand, $\K_{P_1}$ is trivially $E^{\mathrm{ce}}_{P_1}$-learnable, since all its members lie in one $E^{\mathrm{ce}}_{P_1}$-class. Hence $E^{\mathrm{ce}}_{P_1}\not\leq_{\mathrm{Learn}}E^{\mathrm{ce}}_{P_2}$, proving the claim.

Finally, note that $=^{\mathrm{ce}}\subseteq E^{\mathrm{ce}}_P\subseteq F^{\mathrm{ce}}$ for every finite $P$. Indeed, $E^{\mathrm{ce}}_P$ extends $=^{\mathrm{ce}}$ by definition, and assumption~(b) ensures that every member of $\K_P$ is $F^{\mathrm{ce}}$-equivalent to $W_e$. Thus $=^{\mathrm{ce}}\leq_{\mathrm{Learn}}E^{\mathrm{ce}}_P\leq_{\mathrm{Learn}}F^{\mathrm{ce}}$.

If $P=\emptyset$, then $\K_P=\{W_e\}$ and $E^{\mathrm{ce}}_P$ coincides with $=^{\mathrm{ce}}$. For any finite $P$, choose a path $\pi\in[T]\setminus P$. Let $\K_\pi$ be the family consisting of $W_e$ together with all $W_{\ell(\pi\restriction n)}$, $n\in\omega$. By assumption~(b), the family $\K_\pi$ is $F^{\mathrm{ce}}$-learnable. Exactly as above, after discarding a finite initial segment of $\pi$ we obtain a canonical witness sequence outside $\K_P$, so $W_e$ is $(\K_\pi,E^{\mathrm{ce}}_P)$-absorbing. Hence $\K_\pi$ is not $E^{\mathrm{ce}}_P$-learnable. Therefore $E^{\mathrm{ce}}_P<_{\mathrm{Learn}}F^{\mathrm{ce}}$.

Define $\iota(P):=[E^{\mathrm{ce}}_P]_{\equiv_{\mathrm{Learn}}}$. By the claim and the final observations, this is the required embedding.
\end{proof}

\begin{theorem}
\label{thm:atomless}
Let $E^{\mathrm{ce}}$ be an equivalence relation. If $=^{\mathrm{ce}}<_{\mathrm{Learn}}E^{\mathrm{ce}}$, then there is an embedding
\[
\iota\colon(\mathcal P_{\mathrm{fin}}(\Cantor),\subseteq)\to(\mathcal E^{\mathrm{ce}},\leq)
\]
such that:
\begin{itemize}
    \item $\iota(\emptyset)=[=^{\mathrm{ce}}]_{\equiv_{\mathrm{Learn}}}$,
    \item for every nonempty $P\in\mathcal P_{\mathrm{fin}}(\Cantor)$, we have $\iota(P)<[E^{\mathrm{ce}}]_{\equiv_{\mathrm{Learn}}}$.
\end{itemize}
In particular, $\TLH$ is atomless.
\end{theorem}

\begin{proof}
By \Cref{more_learning_power_class_collapsed}, there is a family $\K$ that is not $=^{\mathrm{ce}}$-learnable and whose members all lie in one $E^{\mathrm{ce}}$-equivalence class. By \Cref{prop:eqce_learnable_iff_no_absorbing}, fix a $(\K,=^{\mathrm{ce}})$-absorbing set $W_e$ and a canonical witness sequence for it. By \Cref{obs:canonical_witnesses_basic_facts}, we may pass to an infinite subsequence of pairwise $=^{\mathrm{ce}}$-inequivalent witnesses; all of them remain $E^{\mathrm{ce}}$-equivalent to $e$. Applying \Cref{lem:downward_embed_pfin_cantor} yields the desired embedding.
\end{proof}

The atomless direction works by embedding \emph{below} a given non-bottom degree. For coatomlessness we need the dual construction: starting from a non-universal degree, we build many distinct refinements that still remain strictly below the top.

\begin{lemma}
\label{lem:upward_embed_pfin_cantor}
Let $E^{\mathrm{ce}}$ and $F^{\mathrm{ce}}$ be equivalence relations, and suppose
that $F^{\mathrm{ce}}$ is universal. Let $\K$ be a family of nonempty c.e.\ sets,
let $W_e\in\K$, and let $(W_{u_i})_{i\in\omega}$ be a sequence of
canonical $(\K,E^{\mathrm{ce}})$-absorbing witnesses for $W_e$ such that
$u_i\nEce{E}u_j$ whenever $i\neq j$. Then there is an embedding $\iota$ from $(\mathcal P_{\mathrm{fin}}(\Cantor),\subseteq)$ to $(\mathcal E^{\mathrm{ce}},\leq)$ such that:
\begin{itemize}
    \item $\iota(\emptyset)=[E^{\mathrm{ce}}]_{\equiv_{\mathrm{Learn}}}$,
    \item for every $P\in\mathcal P_{\mathrm{fin}}(\Cantor)$, we have $\iota(P)<[F^{\mathrm{ce}}]_{\equiv_{\mathrm{Learn}}}$.
\end{itemize}
\end{lemma}

\begin{proof}
Let $T:=2^{<\omega}$, and label its nodes exactly as in the proof of \Cref{lem:downward_embed_pfin_cantor}; let $\ell(\sigma)$ denote the label of $\sigma$. For each path $\pi\in[T]$, the sequence $(W_{\ell(\pi\restriction n)})_{n\in\omega}$ is an infinite subsequence of the original witness sequence and is therefore again canonical for $W_e$.

For finite $P\subseteq[T]$, let $\K_P$ be the family consisting of $W_e$ together with all sets of the form $W_{\ell(\pi\restriction n)}$ where $\pi\in P$ and $n\in\omega$. Define $E^{\mathrm{ce}}_P$ by
\[
j\mathrel{E^{\mathrm{ce}}_P}k
\iff
j\mathrel{E^{\mathrm{ce}}}k\ \lor\ \exists h,h'\in\omega\,\bigl(W_h,W_{h'}\in\K_P\land j\mathrel{E^{\mathrm{ce}}}h\land k\mathrel{E^{\mathrm{ce}}}h'\bigr).
\]
Thus $E^{\mathrm{ce}}_P$ is obtained from $E^{\mathrm{ce}}$ by collapsing into one class all $E^{\mathrm{ce}}$-classes of members of $\K_P$.

Exactly as in the proof of \Cref{lem:downward_embed_pfin_cantor}, one checks that for all finite $P_1,P_2\subseteq[T]$,
\[
P_1\subseteq P_2\quad\Longleftrightarrow\quad E^{\mathrm{ce}}_{P_1}\leq_{\mathrm{Learn}}E^{\mathrm{ce}}_{P_2}.
\]
The only extra point in the converse direction is to observe that if $P_1\not\subseteq P_2$ and $\pi\in P_1\setminus P_2$, then for all sufficiently large $n$ the set $W_{\ell(\pi\restriction n)}$ is $E^{\mathrm{ce}}_{P_2}$-inequivalent to $W_e$: it is not in $\K_{P_2}$, it is not $E^{\mathrm{ce}}$-equivalent to any other witness by the pairwise $E^{\mathrm{ce}}$-inequivalence of the witnesses, and since it is an absorbing witness for $W_e$ it is certainly not $E^{\mathrm{ce}}$-equivalent to $W_e$. Hence the corresponding tail witnesses that $W_e$ is $(\K_{P_1},E^{\mathrm{ce}}_{P_2})$-absorbing.

For every finite $P$, the relation $E^{\mathrm{ce}}_P$ is a quotient of $E^{\mathrm{ce}}$, so $E^{\mathrm{ce}}\leq_{\mathrm{Learn}}E^{\mathrm{ce}}_P$. Since $F^{\mathrm{ce}}$ is universal, also $E^{\mathrm{ce}}_P\leq_{\mathrm{Learn}}F^{\mathrm{ce}}$.

To see that the second reduction is strict, fix finite $P\subseteq[T]$, choose a path $\pi\in[T]\setminus P$, and define $\K_\pi$ as before. Because $F^{\mathrm{ce}}$ is universal, the family $\K_\pi$ is $F^{\mathrm{ce}}$-learnable. But after discarding a finite initial segment of $\pi$, the corresponding labels all lie outside $\K_P$ and remain $E^{\mathrm{ce}}_P$-inequivalent to $W_e$, so $W_e$ is $(\K_\pi,E^{\mathrm{ce}}_P)$-absorbing. Therefore $\K_\pi$ is not $E^{\mathrm{ce}}_P$-learnable. Hence $E^{\mathrm{ce}}_P<_{\mathrm{Learn}}F^{\mathrm{ce}}$.

If $P=\emptyset$, then $\K_P=\{W_e\}$ and therefore $E^{\mathrm{ce}}_P=E^{\mathrm{ce}}$. Defining $\iota(P):=[E^{\mathrm{ce}}_P]_{\equiv_{\mathrm{Learn}}}$ gives the required embedding.
\end{proof}

The next two lemmas prepare the coatomless argument by simplifying the shape of a nonlearnable family.

\begin{lemma}
\label{non_learnable_exactly_one_inf_set}
If an equivalence relation $E^{\mathrm{ce}}$ is not universal, then there exists a family $\K$ of nonempty c.e.\ sets such that $\K$ is not $E^{\mathrm{ce}}$-learnable and exactly one member of $\K$ is infinite.
\end{lemma}

\begin{proof}
Since $E^{\mathrm{ce}}$ is not universal, \Cref{BC_characterized_by_LP} implies that $E^{\mathrm{ce}}$ does not have the locking property. Hence there exist an index $e\in\omega$ and a sequence $(b_i)_{i\in\omega}$ such that:
\begin{itemize}
    \item each $W_{b_i}$ is finite,
    \item $W_{b_i}\subseteq W_{b_{i+1}}$ for all $i$,
    \item $\bigcup_i W_{b_i}=W_e$,
    \item for every $k$ there exists $i\geq k$ with $b_i\nEce{E}e$.
\end{itemize}
Let $\K:=\{W_{b_i}:i\in\omega,\ W_{b_i}\neq\emptyset\}\cup\{W_e\}$. Then exactly one member of $\K$ is infinite, namely $W_e$: if $W_e$ were finite, then eventually $W_{b_i}=W_e$, contradicting the last bullet.

The same argument used in the proof of \Cref{BC_characterized_by_LP}, choosing the witnessing approximation sufficiently far along so that it is nonempty, shows that $W_e$ is $(\K,E^{\mathrm{ce}})$-absorbing. Hence, by \Cref{prop:eqce_learnable_iff_no_absorbing}, the family $\K$ is not $E^{\mathrm{ce}}$-learnable.
\end{proof}

\begin{lemma}
\label{non_universal_wlog_has_singletons_witness}
If an equivalence relation $E^{\mathrm{ce}}$ is not universal, then there exist an equivalence relation $F^{\mathrm{ce}}\equiv_{\mathrm{Learn}}E^{\mathrm{ce}}$, a family $\widehat\K$ of nonempty c.e.\ sets, a set $W_e\in\widehat\K$, and a canonical $(\widehat\K,F^{\mathrm{ce}})$-absorbing witness sequence $(W_{u_i})_{i\in\omega}$ for $W_e$ such that the sets $W_e,W_{u_0},W_{u_1},\dots$ are pairwise $F^{\mathrm{ce}}$-inequivalent.
\end{lemma}

\begin{proof}
By \Cref{non_learnable_exactly_one_inf_set}, there is a non-$E^{\mathrm{ce}}$-learnable family with exactly one infinite member. By \Cref{prop:eqce_learnable_iff_no_absorbing}, we may therefore fix an infinite set $W_e$ and a sequence $(W_{u_i})_{i\in\omega}$ of canonical $(\K,E^{\mathrm{ce}})$-absorbing witnesses for $W_e$, all of which are finite.

By \Cref{obs:canonical_witnesses_basic_facts}, the sequence contains infinitely many pairwise distinct sets. Passing to an infinite subsequence if necessary, we may assume that the sets $W_{u_i}$ are pairwise $=^{\mathrm{ce}}$-inequivalent. Let
\[
S:=\{W_{u_i}:i\in\omega\}.
\]
Define an equivalence relation $F^{\mathrm{ce}}$ by
\[
j\mathrel{F^{\mathrm{ce}}}k
\iff
j=^{\mathrm{ce}}k\ \lor\ \bigl(j\mathrel{E^{\mathrm{ce}}}k\text{ and }W_j,W_k\notin S\bigr).
\]
In other words, $F^{\mathrm{ce}}$ is obtained from $E^{\mathrm{ce}}$ by splitting off the $=^{\mathrm{ce}}$-classes of all sets in $S$.

Clearly $F^{\mathrm{ce}}\subseteq E^{\mathrm{ce}}$, so by \Cref{lem:quotienting_increases_learning_power} we have $F^{\mathrm{ce}}\leq_{\mathrm{Learn}}E^{\mathrm{ce}}$. It remains to prove the converse.

We argue by contraposition. Let $\K'$ be a family that is not $F^{\mathrm{ce}}$-learnable. By \Cref{prop:eqce_learnable_iff_no_absorbing}, fix a $(\K',F^{\mathrm{ce}})$-absorbing set $W_f$ together with a canonical witness sequence $(W_{v_i})_{i\in\omega}$ for $W_f$.

There are now two cases.

\smallskip
\noindent\emph{Case 1: infinitely many of the $W_{v_i}$ are not in $S$.} Then pass to an infinite subsequence all of whose members lie outside $S$. Since $W_f$ is infinite and each member of $S$ is finite, we also have $W_f\notin S$. By the definition of $F^{\mathrm{ce}}$, whenever $W_{v_i}\notin S$ and $W_f\notin S$, the implication
\[
v_i\mathrel{E^{\mathrm{ce}}}f\ \Rightarrow\ v_i\mathrel{F^{\mathrm{ce}}}f
\]
holds. Therefore every witness in the subsequence is still $E^{\mathrm{ce}}$-inequivalent to $f$. By \Cref{obs:canonical_witnesses_basic_facts}, the subsequence remains canonical, so $W_f$ is $(\K',E^{\mathrm{ce}})$-absorbing. Hence $\K'$ is not $E^{\mathrm{ce}}$-learnable.

\smallskip
\noindent\emph{Case 2: eventually all $W_{v_i}$ lie in $S$.} Then there is some $j$ such that $W_{v_k}\in S$ for all $k\geq j$. We claim that $W_f=W_e$.

First, $W_f\subseteq W_e$: let $n\in W_f$. For all sufficiently large $k\geq j$, the witness $W_{v_k}$ contains $n$, because the witness sequence is canonical for $W_f$. Since each such $W_{v_k}$ lies in $S$, it is a subset of $W_e$. Hence $n\in W_e$.

Conversely, $W_e\subseteq W_f$: because the tail $(W_{v_k})_{k\geq j}$ contains infinitely many pairwise distinct members of $S$, there is an infinite set of indices $I$ such that $(W_{u_i})_{i\in I}$ occurs along that tail. By \Cref{obs:canonical_witnesses_basic_facts}, this subsequence is still canonical for $W_e$. Since every one of its members appears among the witnesses for $W_f$, they are all subsets of $W_f$. Now fix $n\in W_e$. Some set in this canonical subsequence contains $W_e\restriction(n+1)$, hence contains $n$; therefore $n\in W_f$.

So indeed $W_f=W_e$. Now every witness $W_{v_k}$ with $k\geq j$ belongs to $S$, say $W_{v_k}=W_{u_i}$ for some $i$. Since $(W_{u_i})_{i\in\omega}$ is canonical for $W_e$, we have $u_i\nEce{E}e$. As $W_{v_k}=W_{u_i}$ and $W_f=W_e$, it follows that $v_k\nEce{E}f$. Hence the tail $(W_{v_k})_{k\geq j}$ witnesses that $W_f$ is $(\K',E^{\mathrm{ce}})$-absorbing, and again $\K'$ is not $E^{\mathrm{ce}}$-learnable.

In both cases, every non-$F^{\mathrm{ce}}$-learnable family is non-$E^{\mathrm{ce}}$-learnable. Thus
\[
E^{\mathrm{ce}}\leq_{\mathrm{Learn}}F^{\mathrm{ce}},
\]
and therefore $E^{\mathrm{ce}}\equiv_{\mathrm{Learn}}F^{\mathrm{ce}}$.

Finally, let
\[
\widehat\K:=\{W_e\}\cup\{W_{u_i}:i\in\omega\}.
\]
Since each $W_{u_i}$ is a proper subset of $W_e$ and is $F^{\mathrm{ce}}$-inequivalent to $W_e$, the sequence $(W_{u_i})_{i\in\omega}$ is a canonical $(\widehat\K,F^{\mathrm{ce}})$-absorbing witness sequence for $W_e$. Hence $\widehat\K$ is not $F^{\mathrm{ce}}$-learnable. Moreover, the members of $\widehat\K$ are pairwise $F^{\mathrm{ce}}$-inequivalent: the $W_{u_i}$ have been split into distinct $F^{\mathrm{ce}}$-classes, while $W_e$ is infinite and each $W_{u_i}$ is finite.
\end{proof}

\begin{theorem}
\label{thm:coatomless}
If $E^{\mathrm{ce}}<_{\mathrm{Learn}}\Id_1$ (equivalently, if $E^{\mathrm{ce}}$ is not universal), then there is an embedding
\[
\iota\colon(\mathcal P_{\mathrm{fin}}(\Cantor),\subseteq)\to(\mathcal E^{\mathrm{ce}},\leq)
\]
such that:
\begin{itemize}
    \item $\iota(\emptyset)=[E^{\mathrm{ce}}]_{\equiv_{\mathrm{Learn}}}$,
    \item for every $P\in\mathcal P_{\mathrm{fin}}(\Cantor)$, we have $\iota(P)<[\Id_1]_{\equiv_{\mathrm{Learn}}}$.
\end{itemize}
In particular, $\TLH$ is coatomless.
\end{theorem}

\begin{proof}
By \Cref{non_universal_wlog_has_singletons_witness}, there exists an equivalence relation $F^{\mathrm{ce}}\equiv_{\mathrm{Learn}}E^{\mathrm{ce}}$ together with a family $\widehat\K$, a c.e.\ set $W_e\in\widehat\K$, and a canonical $(\widehat\K,F^{\mathrm{ce}})$-absorbing witness sequence $(W_{u_i})_{i\in\omega}$ for $W_e$ such that $W_e,W_{u_0},W_{u_1},\dots$ are pairwise $F^{\mathrm{ce}}$-inequivalent. Since $E^{\mathrm{ce}}$ is not universal, neither is $F^{\mathrm{ce}}$. Applying \Cref{lem:upward_embed_pfin_cantor} to $F^{\mathrm{ce}}$ and $\Id_1$ yields an embedding with $\iota(\emptyset)=[F^{\mathrm{ce}}]_{\equiv_{\mathrm{Learn}}}$ and $\iota(P)<[\Id_1]_{\equiv_{\mathrm{Learn}}}$ for all finite $P$. Since $[F^{\mathrm{ce}}]_{\equiv_{\mathrm{Learn}}}=[E^{\mathrm{ce}}]_{\equiv_{\mathrm{Learn}}}$, the statement follows.
\end{proof}

For readers' convenience, we summarise the principal structural consequences of this section in the next theorem.

\begin{theorem}
\label{thm:mainstructuraltheorem}
The tolerant learning hierarchy $\TLH$:
\begin{itemize}
    \item is a bounded lower semilattice but not an upper semilattice, with bottom $[=^{\mathrm{ce}}]_{\equiv_{\mathrm{Learn}}}$ and top $[\Id_1]_{\equiv_{\mathrm{Learn}}}$;
    \item is atomless and coatomless;
    \item has width and height $2^{\aleph_0}$.
\end{itemize}
\end{theorem}

The preceding results show that $\TLH$ is both atomless and coatomless. The following natural strengthening remains open.

\begin{question}
Is $\TLH$ dense?
\end{question}

\section{Benchmark relations and their learning power}
\label{sec:comparison}

In this section, we calibrate $\TLH$ against a collection of canonical equivalence relations.  The point of the calibration is not merely to add examples: these relations are c.e.-index analogues of benchmark relations that organise the descriptive-set-theoretic study of classification problems (see, e.g., \cite{gao2008invariant}).  In particular, $E^{\ce}_0$ is the finite-error analogue of eventual agreement, while $E^{\ce}_1$, $E^{\ce}_3$, and $E^{\ce}_{\mathrm{set}}$ are obtained by applying the familiar column constructions to c.e.\ subsets of $\omega\times\omega$.  The relations $E^{\ce}_{\min}$ and $E^{\ce}_{\max}$ provide simple one-dimensional invariants that are useful tests for how learn-reducibility treats coarse semantic information.

The following equivalence relations are defined on $\omega$ and extend $=^{\ce}$.  Whenever we view c.e.\ sets as subsets of $\omega\times\omega$, we use a fixed computable pairing function.  For $W_e\subseteq\omega\times\omega$ and $k\in\omega$, the \define{$k$-th column of $W_e$} is
\[
W_e^{[k]}:=\{i\in\omega:\langle k,i\rangle\in W_e\}.
\]
We consider the following benchmark relations:
\begin{itemize}
    \item $e\mathrel{E^{\mathrm{ce}}_{\min}}e'\iff\bigl(W_e=W_{e'}=\emptyset\bigr)\lor\bigl(W_e,W_{e'}\neq\emptyset\text{ and }\min W_e=\min W_{e'}\bigr)$,
    \item $e\mathrel{E^{\mathrm{ce}}_{\max}}e'\iff\bigl(W_e=W_{e'}=\emptyset\bigr)\lor\bigl(W_e\text{ and }W_{e'}\text{ are both infinite}\bigr)\lor\bigl(W_e,W_{e'}\text{ are finite nonempty and }\max W_e=\max W_{e'}\bigr)$,
    \item $e\mathrel{E^{\mathrm{ce}}_1}e'\iff(\forall^{\infty}k)\,\bigl(W_e^{[k]}=W_{e'}^{[k]}\bigr)$,
    \item $e\mathrel{E^{\mathrm{ce}}_3}e'\iff(\forall k)\,\bigl|W_e^{[k]}\mathbin{\triangle}W_{e'}^{[k]}\bigr|<\infty$,
    \item $e\mathrel{E^{\mathrm{ce}}_{\mathrm{set}}}e'\iff\{W_e^{[k]}:k\in\omega\}=\{W_{e'}^{[k]}:k\in\omega\}$.
\end{itemize}

Thus $E^{\ce}_1$ identifies sets whose columns eventually agree exactly, $E^{\ce}_3$ identifies sets whose corresponding columns agree modulo finite error, and $E^{\ce}_{\mathrm{set}}$ forgets both the order and multiplicity of columns.  These are deliberately close to the standard $E_1$, $E_3$, and Friedman--Stanley set constructions on Baire space \cite{gao2008invariant}; the results below show that their learning-theoretic behaviour is nevertheless quite different from their behaviour under the usual uniform reducibilities.

We already know that $\Id_1$ is universal and $=^{\ce}$ is low. More generally, any equivalence relation whose classes contain only finitely many $=^{\mathrm{ce}}$-classes is low: otherwise, by \Cref{more_learning_power_class_collapsed}, a single such class would contain a non-$=^{\mathrm{ce}}$-learnable family, which is impossible for a finite family.

Our first benchmark is $E^{\mathrm{ce}}_0$, namely equality modulo finite symmetric difference. This is the canonical relation associated with anomalous learning, so it is the most basic test case for how much one gains by allowing finitely many final errors.

\begin{proposition}
$=^{\mathrm{ce}}<_{\mathrm{Learn}}E^{\mathrm{ce}}_0$.
\end{proposition}

\begin{proof}
Since $=^{\mathrm{ce}}$ is the bottom element of $\TLH$, we already have $=^{\mathrm{ce}}\leq_{\mathrm{Learn}}E^{\mathrm{ce}}_0$. To prove strictness, it remains to show
\[
E^{\mathrm{ce}}_0\not\leq_{\mathrm{Learn}}=^{\mathrm{ce}}.
\]
Let
\[
\K:=\{W_e:W_e\text{ is cofinite}\}.
\]
All members of $\K$ lie in one $E^{\mathrm{ce}}_0$-equivalence class, so $\K$ is trivially $E^{\mathrm{ce}}_0$-learnable. On the other hand, $\omega$ is $(\K,=^{\mathrm{ce}})$-absorbing, so by \Cref{prop:eqce_learnable_iff_no_absorbing} the family $\K$ is not $=^{\mathrm{ce}}$-learnable.
\end{proof}

We already noted in the introduction that $E^{\mathrm{ce}}_0$ is not universal. In fact, it is weak in a stronger sense.

\begin{definition}
An equivalence relation $E^{\mathrm{ce}}$ is \define{weak} if every
equivalence relation $F^{\mathrm{ce}}\supseteq E^{\mathrm{ce}}$ that
is universal is equal to $\Id_1$.
\end{definition}

In other words, weakness says that there is no nontrivial way of quotienting $E^{\mathrm{ce}}$ into a universal relation.

\begin{theorem}
$E^{\mathrm{ce}}_0$ is weak.
\end{theorem}

\begin{proof}
Let $F^{\mathrm{ce}}\supseteq E^{\mathrm{ce}}_0$, and suppose that $F^{\mathrm{ce}}$ is universal. We must show that $F^{\mathrm{ce}}=\Id_1$.

All indices of finite sets lie in one $E^{\mathrm{ce}}_0$-equivalence class, hence also in one $F^{\mathrm{ce}}$-equivalence class. Let $A$ denote that class. We claim that every index belongs to $A$.

Let $e\in\omega$. If $W_e$ is finite, then of course $e\in A$. Suppose instead that $W_e$ is infinite and $e\notin A$. Let
\[
\K_e:=\{W_i:W_i\text{ is finite and nonempty}\}\cup\{W_e\}.
\]
Since $F^{\mathrm{ce}}$ is universal, the family $\K_e$ is $F^{\mathrm{ce}}$-learnable.

But $W_e$ is $(\K_e,F^{\mathrm{ce}})$-absorbing. Indeed, given a finite $F\subseteq W_e$, choose a finite nonempty set $W_i$ with $F\subseteq W_i\subsetneq W_e$. Since $W_i$ is finite, we have $i\in A$, while $e\notin A$ by assumption, so $i\nEce{F}e$. Thus every finite subset of $W_e$ extends to a wrong proper subtarget. By \Cref{prop:eqce_learnable_iff_no_absorbing}, the family $\K_e$ is not $F^{\mathrm{ce}}$-learnable, contradiction.

Therefore every index lies in $A$, so $F^{\mathrm{ce}}$ has exactly one equivalence class, i.e., $F^{\mathrm{ce}}=\Id_1$.
\end{proof}

So finite-error tolerance does not bring one robustly close to universality: reaching the top requires a total collapse.

We now compare $E^{\mathrm{ce}}_0$ with the other benchmark relations defined above.

\begin{proposition}
\label{prop:E0belowE1}
We have
\[
E^{\mathrm{ce}}_0\leq_{\mathrm{Learn}}E^{\mathrm{ce}}_1,E^{\mathrm{ce}}_3,E^{\mathrm{ce}}_{\max},
\qquad\text{and}\qquad
E^{\mathrm{ce}}_0\not\leq_{\mathrm{Learn}}E^{\mathrm{ce}}_{\mathrm{set}}.
\]
\end{proposition}

\begin{proof}
Since $E^{\mathrm{ce}}_0\subseteq E^{\mathrm{ce}}_1$ and $E^{\mathrm{ce}}_0\subseteq E^{\mathrm{ce}}_3$, the first two reductions follow immediately from \Cref{lem:quotienting_increases_learning_power}.

For $E^{\mathrm{ce}}_{\max}$, let $\K$ be a family that is $E^{\mathrm{ce}}_0$-learnable via a learner $\learnerM$. Fix a computable function $q\colon\omega\to\omega$ such that $W_{q(m)}=\{m\}$ for every $m\in\omega$. Define a new learner $\learnerM'$ on a finite string $\sigma$ by
\[
\learnerM'(\sigma)=
\begin{cases}
\learnerM(\sigma),
  &\text{if }W_{\learnerM(\sigma)}\text{ is infinite},\\
q(\max(\operatorname{range}(\sigma))),
  &\text{if }W_{\learnerM(\sigma)}\text{ is finite and }\operatorname{range}(\sigma)\neq\emptyset,\\
q(0),&\text{otherwise.}
\end{cases}
\]

If $\mu$ enumerates an infinite set $W_e\in\K$, then $\learnerM$ eventually stabilises to an index of an infinite set, since no finite set is $E^{\mathrm{ce}}_0$-equivalent to $W_e$. From that point on, $\learnerM'$ copies $\learnerM$ and hence converges to an $E^{\mathrm{ce}}_{\max}$-correct hypothesis.

If $\mu$ enumerates a finite set $W_e\in\K$, then the limiting hypothesis of $\learnerM$ indexes a finite set. Hence, after a stage at which both $\learnerM$ has stabilised and the maximum $m$ of $W_e$ has appeared, $\learnerM'$ constantly outputs $q(m)$, which is $E^{\mathrm{ce}}_{\max}$-equivalent to $W_e$. Therefore $\K$ is $E^{\mathrm{ce}}_{\max}$-learnable.

Finally, to show that $E^{\mathrm{ce}}_0\not\leq_{\mathrm{Learn}}E^{\mathrm{ce}}_{\mathrm{set}}$, let
\[
D_i:=\{\langle k,0\rangle:k\neq i\}\quad(i\in\omega),
\qquad
D_\omega:=\{\langle k,0\rangle:k\in\omega\},
\]
and set $\K_D:=\{D_i:i\in\omega\}\cup\{D_\omega\}$. This family is $E^{\mathrm{ce}}_0$-learnable, since any two of its members differ only finitely.

On the other hand, $D_\omega$ is $(\K_D,E^{\mathrm{ce}}_{\mathrm{set}})$-absorbing. Given a finite set $F\subseteq D_\omega$, choose $i$ with $F\subseteq D_i\subsetneq D_\omega$. Then
\[
\{D_i^{[k]}:k\in\omega\}=\{\emptyset,\{0\}\},
\qquad
\{D_\omega^{[k]}:k\in\omega\}=\{\{0\}\},
\]
so $D_i\mathrel{\cancel{E^{\mathrm{ce}}_{\mathrm{set}}}}D_\omega$. Hence $\K_D$ is not $E^{\mathrm{ce}}_{\mathrm{set}}$-learnable.
\end{proof}

The proof of $E^{\mathrm{ce}}_0\leq_{\mathrm{Learn}}E^{\mathrm{ce}}_{\max}$ is non-effective, since it tests whether a hypothesised c.e.\ set is infinite---a $\Pi^0_2$-complete predicate, decidable relative to $0^{\mathord{\prime\prime}}$. This raises the question whether the implication
\[
E^{\mathrm{ce}}_0\leq_{\mathrm{Learn}}E^{\mathrm{ce}}_{\max}
\]
continues to hold when both notions of learnability are restricted to
computable learners.

To obtain a fuller comparison picture, we next consider three simple witness families:
\begin{itemize}
    \item $\K_A:=\{A_i:i\in\omega\}\cup\{A_\omega\}$, where $A_i:=\{\langle 0,k\rangle:k\leq i\}$ and $A_\omega:=\{\langle 0,k\rangle:k\in\omega\}$;
    \item $\K_B:=\{B_i:i\in\omega\}\cup\{B_\omega\}$, where $B_i:=\{\langle k+1,0\rangle:k\leq i\}$ and $B_\omega:=\{\langle k+1,0\rangle:k\in\omega\}$;
    \item $\K_C:=\{C_i:i\in\omega\}\cup\{C_\omega\}$, where $C_i:=\{\langle k+1,j\rangle:k\leq i\text{ and }j\in\omega\}$ and $C_\omega:=\{\langle k+1,j\rangle:k,j\in\omega\}$.
\end{itemize}

\begin{proposition}
\label{prop:KAE1ce}
The family $\K_A$ is $E^{\mathrm{ce}}_1$-learnable but not $E^{\mathrm{ce}}_3$-, $E^{\mathrm{ce}}_{\mathrm{set}}$-, or $E^{\mathrm{ce}}_{\max}$-learnable.
\end{proposition}

\begin{proof}
For every $i\in\omega$ and every $k>0$, we have $A_i^{[k]}=A_\omega^{[k]}=\emptyset$. Hence all members of $\K_A$ are $E^{\mathrm{ce}}_1$-equivalent, so $\K_A$ is $E^{\mathrm{ce}}_1$-learnable.

For the remaining relations, $A_\omega$ is absorbing. Given finite $F\subseteq A_\omega$, choose $i$ such that $F\subseteq A_i\subsetneq A_\omega$.
\begin{itemize}
    \item Since $|A_i^{[0]}\mathbin{\triangle} A_\omega^{[0]}|=\infty$, we have $A_i\mathrel{\cancel{E^{\mathrm{ce}}_3}}A_\omega$.
    \item Since $\{A_i^{[k]}:k\in\omega\}=\{\emptyset,\{0,\dots,i\}\}$ while $\{A_\omega^{[k]}:k\in\omega\}=\{\emptyset,\omega\}$, we have $A_i\mathrel{\cancel{E^{\mathrm{ce}}_{\mathrm{set}}}}A_\omega$.
    \item Each $A_i$ is finite while $A_\omega$ is infinite, so $A_i\mathrel{\cancel{E^{\mathrm{ce}}_{\max}}}A_\omega$.
\end{itemize}
In each case, \Cref{prop:eqce_learnable_iff_no_absorbing} yields nonlearnability.
\end{proof}

\begin{proposition}
\label{prop:KBE1ce}
The family $\K_B$ is $E^{\mathrm{ce}}_3$-learnable and $E^{\mathrm{ce}}_{\mathrm{set}}$-learnable, but not $E^{\mathrm{ce}}_1$- or $E^{\mathrm{ce}}_{\max}$-learnable.
\end{proposition}

\begin{proof}
For any $X,Y\in\K_B$ and every $k\in\omega$, the symmetric difference
$X^{[k]}\mathbin{\triangle}Y^{[k]}$ is finite. Hence all members of $\K_B$ are
$E^{\mathrm{ce}}_3$-equivalent, so $\K_B$ is $E^{\mathrm{ce}}_3$-learnable.

Moreover, for every $i$,
\[
\{B_i^{[k]}:k\in\omega\}=\{\emptyset,\{0\}\},
\]
so all members of $\K_B$ are $E^{\mathrm{ce}}_{\mathrm{set}}$-equivalent, and $\K_B$ is $E^{\mathrm{ce}}_{\mathrm{set}}$-learnable.

For $E\in\{E^{\mathrm{ce}}_1,E^{\mathrm{ce}}_{\max}\}$, the set $B_\omega$ is $(\K_B,E)$-absorbing. Given finite $F\subseteq B_\omega$, choose $i$ with $F\subseteq B_i\subsetneq B_\omega$. Then for all $k>i+1$ we have $B_i^{[k]}=\emptyset\neq\{0\}=B_\omega^{[k]}$, so $B_i\mathrel{\cancel{E^{\mathrm{ce}}_1}}B_\omega$; and since $B_i$ is finite while $B_\omega$ is infinite, also $B_i\mathrel{\cancel{E^{\mathrm{ce}}_{\max}}}B_\omega$. Hence $\K_B$ is not learnable for either relation.
\end{proof}

\begin{proposition}
\label{prop:KDEsetce}
The family $\K_C$ is $E^{\mathrm{ce}}_{\max}$-learnable and $E^{\mathrm{ce}}_{\mathrm{set}}$-learnable, but not $E^{\mathrm{ce}}_1$- or $E^{\mathrm{ce}}_3$-learnable.
\end{proposition}

\begin{proof}
Every member of $\K_C$ is infinite. Hence all members of $\K_C$ are $E^{\mathrm{ce}}_{\max}$-equivalent, so $\K_C$ is $E^{\mathrm{ce}}_{\max}$-learnable.

Also, for every $i$,
\[
\{C_i^{[k]}:k\in\omega\}=\{\emptyset,\omega\},
\]
so all members of $\K_C$ are $E^{\mathrm{ce}}_{\mathrm{set}}$-equivalent, and thus $\K_C$ is $E^{\mathrm{ce}}_{\mathrm{set}}$-learnable.

For $E\in\{E^{\mathrm{ce}}_1,E^{\mathrm{ce}}_3\}$, the set $C_\omega$ is $(\K_C,E)$-absorbing. Given finite $F\subseteq C_\omega$, choose $i$ with $F\subseteq C_i\subsetneq C_\omega$. Then for every $k>i+1$ we have $C_i^{[k]}=\emptyset\neq\omega=C_\omega^{[k]}$, so $C_i\mathrel{\cancel{E^{\mathrm{ce}}_1}}C_\omega$; moreover, for each such $k$ the symmetric difference is infinite, so $C_i\mathrel{\cancel{E^{\mathrm{ce}}_3}}C_\omega$. By \Cref{prop:eqce_learnable_iff_no_absorbing}, $\K_C$ is not learnable for either relation.
\end{proof}

The three witness families immediately yield the following incomparability theorem.

\begin{theorem}
\label{thm:incomparability}
The relations $E^{\mathrm{ce}}_1$, $E^{\mathrm{ce}}_3$, $E^{\mathrm{ce}}_{\mathrm{set}}$, and $E^{\mathrm{ce}}_{\max}$ are pairwise incomparable under $\leq_{\mathrm{Learn}}$.
\end{theorem}

\begin{proof}
By \Cref{prop:KAE1ce}, the family $\K_A$ is $E^{\mathrm{ce}}_1$-learnable but not $E^{\mathrm{ce}}_3$- or $E^{\mathrm{ce}}_{\max}$-learnable. Hence
\[
E^{\mathrm{ce}}_1\not\leq_{\mathrm{Learn}}E^{\mathrm{ce}}_3,E^{\mathrm{ce}}_{\max}.
\]

By \Cref{prop:KBE1ce}, the family $\K_B$ is $E^{\mathrm{ce}}_3$-learnable and $E^{\mathrm{ce}}_{\mathrm{set}}$-learnable, but not $E^{\mathrm{ce}}_1$- or $E^{\mathrm{ce}}_{\max}$-learnable. Hence
\[
E^{\mathrm{ce}}_3\not\leq_{\mathrm{Learn}}E^{\mathrm{ce}}_1,E^{\mathrm{ce}}_{\max},
\qquad
E^{\mathrm{ce}}_{\mathrm{set}}\not\leq_{\mathrm{Learn}}E^{\mathrm{ce}}_1,E^{\mathrm{ce}}_{\max}.
\]

By \Cref{prop:KDEsetce}, the family $\K_C$ is $E^{\mathrm{ce}}_{\max}$-learnable and $E^{\mathrm{ce}}_{\mathrm{set}}$-learnable, but not $E^{\mathrm{ce}}_1$- or $E^{\mathrm{ce}}_3$-learnable. Hence
\[
E^{\mathrm{ce}}_{\max}\not\leq_{\mathrm{Learn}}E^{\mathrm{ce}}_1,E^{\mathrm{ce}}_3,
\qquad
E^{\mathrm{ce}}_{\mathrm{set}}\not\leq_{\mathrm{Learn}}E^{\mathrm{ce}}_1,E^{\mathrm{ce}}_3.
\]

Finally, by \Cref{prop:E0belowE1},
\[
E^{\mathrm{ce}}_0\leq_{\mathrm{Learn}}E^{\mathrm{ce}}_1,E^{\mathrm{ce}}_3,E^{\mathrm{ce}}_{\max}
\qquad\text{and}\qquad
E^{\mathrm{ce}}_0\not\leq_{\mathrm{Learn}}E^{\mathrm{ce}}_{\mathrm{set}}.
\]
By transitivity, this implies that none of $E^{\mathrm{ce}}_1,E^{\mathrm{ce}}_3,E^{\mathrm{ce}}_{\max}$ can reduce to $E^{\mathrm{ce}}_{\mathrm{set}}$.

Combining all these nonreductions shows that the four relations are pairwise incomparable.
\end{proof}

Thus the benchmark relations do not line up in any simple chain under learn-reducibility. In particular, moving to coarser semantic notions can change learning power in genuinely incomparable ways.

\begin{corollary}
We have
\[
E^{\mathrm{ce}}_0<_{\mathrm{Learn}}E^{\mathrm{ce}}_1,\qquad
E^{\mathrm{ce}}_0<_{\mathrm{Learn}}E^{\mathrm{ce}}_3,\qquad
E^{\mathrm{ce}}_0<_{\mathrm{Learn}}E^{\mathrm{ce}}_{\max},
\]
and
\[
E^{\mathrm{ce}}_0
\mathrel{|_{\mathrm{Learn}}}
E^{\mathrm{ce}}_{\mathrm{set}}.
\]
\end{corollary}

\begin{proof}
The reductions from $E^{\mathrm{ce}}_0$ to $E^{\mathrm{ce}}_1$, $E^{\mathrm{ce}}_3$, and $E^{\mathrm{ce}}_{\max}$ follow from \Cref{prop:E0belowE1}. They are strict because $A_\omega$, $B_\omega$, and $C_\omega$ are, respectively, $(\K_A,E^{\mathrm{ce}}_0)$-, $(\K_B,E^{\mathrm{ce}}_0)$-, and $(\K_C,E^{\mathrm{ce}}_0)$-absorbing: every finite subset of the distinguished set is contained in a proper member of the corresponding family whose symmetric difference from it is infinite. Hence \Cref{prop:KAE1ce,prop:KBE1ce,prop:KDEsetce} witness, respectively,
\[
E^{\mathrm{ce}}_1\not\leq_{\mathrm{Learn}}E^{\mathrm{ce}}_0,\qquad
E^{\mathrm{ce}}_3\not\leq_{\mathrm{Learn}}E^{\mathrm{ce}}_0,\qquad
E^{\mathrm{ce}}_{\max}\not\leq_{\mathrm{Learn}}E^{\mathrm{ce}}_0.
\]
The final incomparability is exactly the last part of \Cref{prop:E0belowE1}, together with the fact that $E^{\mathrm{ce}}_{\mathrm{set}}\not\leq_{\mathrm{Learn}}E^{\mathrm{ce}}_0$, witnessed by $\K_B$ (or by $\K_C$).
\end{proof}

We conclude this section with a natural universal relation distinct from $\Id_1$.

\begin{proposition}
$E^{\mathrm{ce}}_{\min}$ is universal.
\end{proposition}

\begin{proof}
By \Cref{prop:eqce_learnable_iff_no_absorbing}, it suffices to show that there is no $E^{\mathrm{ce}}_{\min}$-absorbing set. Suppose towards a contradiction that $W_i$ is $E^{\mathrm{ce}}_{\min}$-absorbing. The set $W_i$ cannot be finite, since taking $F=W_i$ would leave no proper subset of $W_i$ containing $F$. Hence $W_i$ is infinite, in particular nonempty, and has a least element $n$. Consider the finite set $\{n\}\subseteq W_i$. Any c.e.\ set $W_u$ satisfying $\{n\}\subseteq W_u\subsetneq W_i$ must also have least element $n$, and therefore $u\mathrel{E^{\mathrm{ce}}_{\min}}i$. So $\{n\}$ has no $E^{\mathrm{ce}}_{\min}$-absorbing witness, contradiction.
\end{proof}

\subsection{Uniformity and comparison with other reducibilities}

The preorder $\leq_{\mathrm{Learn}}$ is very different from the standard reducibilities used in descriptive set theory and computability theory. To avoid confusion, we briefly compare the same benchmark relations under Borel reducibility, continuous reducibility, computable reducibility, and learn-reducibility; see Table~\ref{picture:summary}.

In general, a \define{uniform reduction} from an equivalence relation $E$ on a space $X$ to an equivalence relation $F$ on a space $Y$ is a function $\Gamma$ such that
$x\, E\, x'\iff\Gamma(x)\, F\, \Gamma(x')$.
If no definability restriction is imposed on $\Gamma$, reducibility depends only on the number of equivalence classes and is therefore too coarse to be informative. In descriptive set theory the most studied notion is \define{Borel reducibility}, denoted $\leq_B$, for equivalence relations on Polish spaces. One can also require the reduction to be continuous, yielding \define{continuous reducibility}, denoted $\leq_{\mathrm{cont}}$. In computability theory one usually works with equivalence relations on $\omega$ and asks the reduction to be computable; this yields \define{computable reducibility}, denoted
$\leq_{\mathrm{comp}}$.

The first two panels of Table~\ref{picture:summary}
combine the standard Borel and continuous comparisons among the classical
benchmark relations $=,E_0,E_1,E_3$, and $E_{\mathrm{set}}$ with the
elementary placements of $E_{\min}$ and $E_{\max}$. For background on
these benchmark relations and their Borel reducibility ordering, see
\cite{gao2008invariant}. The computable-reducibility
panel summarises standard comparisons among the corresponding c.e.-index
relations; see \cite{FokinaFriedmanNies2012Sigma3,coskey2012hierarchy}.
The next proposition records the elementary Borel and continuous comparisons
involving $E_{\min}$ and $E_{\max}$ that are needed to complete the first
two panels. The final panel collects the learn-reducibility results
established above. In all four panels, lines are read from bottom to top
as strict reductions, and transitive edges are omitted.

\smallskip

In this subsection, the unadorned benchmark relations are the standard relations on Baire space. Thus $=$ is equality on $\omega^\omega$, $=_1$ is the one-class equivalence relation, and $E_0$ is eventual equality. For $m,t\in\omega$, write $m^t$ for the length-$t$ constant string with value $m$ and $m^\omega$ for the constant real with value $m$. For $x\in\omega^\omega$, write $\operatorname{range}(x):=\{x(n):n\in\omega\}$ and $x^{[k]}(n):=x(\langle k,n\rangle)$. We also let
\begin{align*}
x\mathrel{E_1}y
&\iff (\forall^\infty k)\,\bigl(x^{[k]}=y^{[k]}\bigr),\\
x\mathrel{E_3}y
&\iff (\forall k)\,\bigl(x^{[k]}\mathrel{E_0}y^{[k]}\bigr),\\
x\mathrel{E_{\mathrm{set}}}y
&\iff \{x^{[k]}:k\in\omega\}=\{y^{[k]}:k\in\omega\},\\
x\mathrel{E_{\min}}y
&\iff \min(\operatorname{range}(x))=\min(\operatorname{range}(y)),\\
x\mathrel{E_{\max}}y
&\iff \bigl(x,y\text{ have no maximum in their ranges}\bigr)\\
&\hspace{2.4em}{}\lor
\bigl(x,y\text{ have maxima and }
\max(\operatorname{range}(x))=\max(\operatorname{range}(y))\bigr).
\end{align*}

\medskip

\begin{table}[H]
\scalebox{0.9}{
    \begin{tabular}{cccc}
    \input{BorelRed.tikz} &
    \input{contRed.tikz} &
    \input{EceHierarchy.tikz} &
    \input{TLH.tikz}
    \\\hline
    \end{tabular}
}
\caption{From left to right: the benchmark relations ordered by Borel reducibility, continuous reducibility, computable reducibility, and learn-reducibility.}
\label{picture:summary}
\end{table}

\bigskip

\begin{proposition}\label{prop:uniformity_comparisons}
The following hold:
\begin{itemize}
    \item[(a)] $=_1<_B E_{\min}\equiv_B E_{\max}<_B =$;
    \item[(b)] $=_1<_{\mathrm{cont}}R$ for every
    $R\in\{=,E_{\min},E_{\max}\}$, and the relations
    $=$, $E_{\min}$, and $E_{\max}$ are pairwise incomparable under
    $\leq_{\mathrm{cont}}$;
    \item[(c)] $E_{\min}<_{\mathrm{cont}}E_0$, while
    $E_{\max}\nleq_{\mathrm{cont}}E_0$;
    \item[(d)] $E_{\max}<_{\mathrm{cont}}E_3$.
\end{itemize}
\end{proposition}

\begin{proof}
For (a), note first that $=_1$ has only one equivalence class, whereas
$E_{\min}$, $E_{\max}$, and $=$ have more than one. The relations
$E_{\min}$ and $E_{\max}$ are smooth and have countably many equivalence
classes; each is therefore Borel bireducible with equality on $\omega$,
which is strictly below equality on Baire space.

For (b), constant maps show that $=_1\leq_{\mathrm{cont}}R$ for every
$R\in\{=,E_{\min},E_{\max}\}$, and these reductions are strict because
all three target relations have more than one equivalence class. Equality
on $\omega^\omega$ is not reducible to either $E_{\min}$ or $E_{\max}$,
even without a regularity requirement on the reduction, because it has
continuum many equivalence classes whereas $E_{\min}$ and $E_{\max}$
have only countably many.

Suppose that a continuous map $\Gamma$ reduces $E_{\min}$ to equality.
Since $1^\omega\mathrel{\cancel{E_{\min}}}2^\omega$, choose $q$ such that
$\Gamma(1^\omega)(q)\neq\Gamma(2^\omega)(q)$. By continuity at
$2^\omega$, there is a $t$ such that every $x\sqsupseteq 2^t$ satisfies
$\Gamma(x)(q)=\Gamma(2^\omega)(q)$. But $2^t{}^\smallfrown 1^\omega$ is
$E_{\min}$-equivalent to $1^\omega$, a contradiction.

Similarly, suppose that a continuous map $\Gamma$ reduces $E_{\max}$ to
equality. Choose $q$ such that
$\Gamma(1^\omega)(q)\neq\Gamma(2^\omega)(q)$. By continuity at
$1^\omega$, there is a $t$ such that the $q$th output coordinate is fixed
on all extensions of $1^t$. The real $1^t{}^\smallfrown 2^\omega$ is
$E_{\max}$-equivalent to $2^\omega$, but its $q$th output coordinate is
forced to agree with that of $\Gamma(1^\omega)$, again a contradiction.

It remains to compare $E_{\min}$ and $E_{\max}$. Suppose first that
$\Gamma$ continuously reduces $E_{\min}$ to $E_{\max}$. The constant
reals $n^\omega$ are pairwise $E_{\min}$-inequivalent, so their images
are pairwise $E_{\max}$-inequivalent. At most one image has no maximum.
For every other $n$, let $a_n$ be the maximum of
$\operatorname{range}(\Gamma(n^\omega))$. The values $a_n$ are pairwise
distinct, so choose $p<q$ such that both are defined and $a_p<a_q$.
Choose a coordinate $r$ with $\Gamma(q^\omega)(r)=a_q$. By continuity
at $q^\omega$, there is a $t$ such that every $x\sqsupseteq q^t$ satisfies
$\Gamma(x)(r)=a_q$. Now $q^t{}^\smallfrown p^\omega\mathrel{E_{\min}}p^\omega$, so its
image must have maximum $a_p$; however, that image takes the value
$a_q>a_p$ at coordinate $r$, a contradiction.

Conversely, suppose that $\Gamma$ continuously reduces $E_{\max}$ to
$E_{\min}$. For each $n$, let
$b_n:=\min(\operatorname{range}(\Gamma(n^\omega)))$. Since the constant
reals are pairwise $E_{\max}$-inequivalent, the values $b_n$ are pairwise
distinct. Choose $p<q$ with $b_p<b_q$, and choose a coordinate $r$ with
$\Gamma(p^\omega)(r)=b_p$. By continuity at $p^\omega$, there is a $t$
such that every $x\sqsupseteq p^t$ satisfies
$\Gamma(x)(r)=b_p$. But $p^t{}^\smallfrown q^\omega\mathrel{E_{\max}}q^\omega$, so its
image must have minimum $b_q$, while it contains the smaller value
$b_p$, a contradiction. This completes the proof of (b).

For (c), define $\Gamma\colon\omega^\omega\to\omega^\omega$ by
\[
\Gamma(p)(s):=\min\{p(j):j\leq s\}.
\]
This map is continuous, and its values are eventually constant at
$\min(\operatorname{range}(p))$. Hence it reduces $E_{\min}$ to $E_0$.
The reduction is strict because $E_0$ has continuum many equivalence
classes whereas $E_{\min}$ has only countably many.

We next show that $E_{\max}\nleq_{\mathrm{cont}}E_0$. Suppose that
$\Gamma$ were such a reduction. We recursively construct finite strings
$\sigma_n,\tau_n$ and pairwise distinct coordinates $k_n$ such that
\begin{itemize}
    \item $\sigma_n\sqsubseteq\sigma_{n+1}$ and
    $\tau_n\sqsubseteq\tau_{n+1}$;
    \item $\max(\operatorname{range}(\sigma_n))=n$ and
    $\max(\operatorname{range}(\tau_n))=n+1$;
    \item every $x\sqsupseteq\sigma_{n+1}$ and
    $y\sqsupseteq\tau_{n+1}$ satisfy
    $\Gamma(x)(k_n)\neq\Gamma(y)(k_n)$.
\end{itemize}
Start with $\sigma_0=\langle0\rangle$ and $\tau_0=\langle1\rangle$.
Given $\sigma_n$ and $\tau_n$, let
\[
x_n:=\sigma_n^\smallfrown(n+1)^\omega,
\qquad
y_n:=\tau_n^\smallfrown(n+2)^\omega.
\]
The two reals have different maxima, so their images are not
$E_0$-equivalent and therefore disagree at infinitely many coordinates.
Choose a fresh such coordinate $k_n$. By continuity, extend $\sigma_n$
and $\tau_n$ to prefixes $\sigma_{n+1}\sqsubseteq x_n$ and
$\tau_{n+1}\sqsubseteq y_n$ that force this disagreement and contain,
respectively, the values $n+1$ and $n+2$. Let
$x:=\bigcup_n\sigma_n$ and $y:=\bigcup_n\tau_n$. Neither real has a
maximum, so $x\mathrel{E_{\max}}y$, while their images differ at every
coordinate $k_n$ and are therefore not $E_0$-equivalent, a contradiction.

For (d), define $\Gamma\colon\omega^\omega\to\omega^\omega$ by
\[
\Gamma(p)(\langle i,s\rangle)=
\begin{cases}
1,&\text{if }(\exists j\leq s)\,p(j)>i,\\
0,&\text{otherwise.}
\end{cases}
\]
For each fixed $i$, the $i$th column of $\Gamma(p)$ is eventually $1$
exactly when some value of $p$ exceeds $i$, and is identically $0$
otherwise. Thus two reals have the same maximum, or both have no maximum,
exactly when their images are $E_3$-equivalent. Hence
$E_{\max}\leq_{\mathrm{cont}}E_3$. The reduction is strict because
$E_3$ has continuum many equivalence classes whereas $E_{\max}$ has only
countably many.
\end{proof}

\section{Computable universality}
\label{section:effective_learners}

All learners considered so far were arbitrary functions on finite strings.  This is the standard non-effective version of the paradigm and is the right level for the structural results above.  It is nevertheless natural to ask whether the universal tolerance relations isolated by the locking property remain universal when learners are required to be computable.  The answer is no: the semantic coarse-graining supplied by a universal relation may depend on noncomputable information that no computable learner can exploit uniformly.

We first isolate the effective notion explicitly.

\begin{definition}
Let $E^{\mathrm{ce}}$ be an equivalence relation on $\omega$. A family $\K$ of nonempty c.e.\ sets is \define{computably $E^{\mathrm{ce}}$-learnable} if there exists a computable learner $\learnerM\colon\omega^{<\omega}\to\omega$ such that for every $W_j\in\K$ and every enumeration $\mu$ of $W_j$ there is some $i\in\omega$ with
\[
\lim_{s\to\infty}\learnerM(\mu[s])=i
\qquad\text{and}\qquad
i\mathrel{E^{\mathrm{ce}}}j.
\]
An equivalence relation $E^{\mathrm{ce}}$ is \define{computably universal} if every family of nonempty c.e.\ sets is computably $E^{\mathrm{ce}}$-learnable; equivalently, if $\{W_e:e\in\omega,\ W_e\neq\emptyset\}$ is computably $E^{\mathrm{ce}}$-learnable.
\end{definition}

A full analysis of computable $E^{\mathrm{ce}}$-learnability lies beyond the scope of this paper. Our goal here is simply to show that computable universality is strictly stronger than universality.

\begin{definition}
For every set $A\subseteq\omega$, define an equivalence relation $E^{\mathrm{ce}}_A$ by
\[
i\mathrel{E^{\mathrm{ce}}_A}j
\iff
\bigl(W_i,W_j\subseteq A\bigr)\ \lor\ \bigl(W_i,W_j\not\subseteq A\bigr).
\]
\end{definition}

\begin{proposition}
\label{inf_subset_eqs_have_LP}
For every $A\subseteq\omega$, the relation $E^{\mathrm{ce}}_A$ is universal.
\end{proposition}

\begin{proof}
By \Cref{BC_characterized_by_LP}, it suffices to show that $E^{\mathrm{ce}}_A$ has the locking property. Let $W_e$ be a nonempty c.e.\ set, and let $(b_i)_{i\in\omega}$ be an increasing sequence of finite approximations whose union is $W_e$.

If $W_e\subseteq A$, then every $W_{b_i}\subseteq A$ as well, so $b_i\mathrel{E^{\mathrm{ce}}_A}e$ for all $i$; thus $0$ is a locking witness.

If $W_e\not\subseteq A$, choose $x\in W_e\setminus A$. Since the approximations exhaust $W_e$, there is some $k$ with $x\in W_{b_k}$. Then for every $j\geq k$, the set $W_{b_j}$ is not contained in $A$. Since also $W_e\not\subseteq A$, we have $b_j\mathrel{E^{\mathrm{ce}}_A}e$ for all $j\geq k$. So $k$ is a locking witness.
\end{proof}

In what follows, let $K:=\{e\in\omega:\varphi_e(e)\downarrow\}$ be the halting set. It is well known that $K$ is c.e.\ but not computable.

The halting set provides a natural test case for the gap between non-effective and effective tolerant learning. The relation $E^{\mathrm{ce}}_K$ is coarse enough to be universal in the abstract sense, but the noncomputability of $K$ prevents any computable learner from exploiting that coarse semantics uniformly.

\begin{theorem}
\label{thm:EK_not_comp_universal}
The relation $E^{\mathrm{ce}}_K$ is universal but not computably universal.
\end{theorem}

\begin{proof}
By \Cref{inf_subset_eqs_have_LP}, the relation $E^{\mathrm{ce}}_K$ is universal.

Suppose towards a contradiction that $E^{\mathrm{ce}}_K$ is computably universal, witnessed by a computable learner $\learnerM$. Fix an index $k$ such that $W_k=K$.

By \Cref{thm:learnability_implies_lock_seq}, there is an $E^{\mathrm{ce}}_K$-locking sequence $\sigma$ for $\learnerM$ on $K$. Thus:
\begin{itemize}
    \item $\operatorname{range}(\sigma)\subseteq K$,
    \item if $\operatorname{range}(\tau)\subseteq K$, then $\learnerM(\sigma^\smallfrown\tau)=\learnerM(\sigma)$,
    \item $\learnerM(\sigma)\mathrel{E^{\mathrm{ce}}_K}k$.
\end{itemize}
Let $\kappa$ be a total computable enumeration of $K$ (with repetitions allowed), and for each $s\in\omega$ let $\tau_s:=\kappa[s]$.

Hard-coding the finite string $\sigma$, we now define a computable enumeration procedure. At stage $t$, for every $a,s\leq t$, compute $\learnerM(\sigma^\smallfrown\langle a\rangle^\smallfrown\tau_s)$ and compare it with $\learnerM(\sigma)$. If they are different and $a$ has not yet been printed, print $a$.

This procedure enumerates a c.e.\ set. We claim that it enumerates exactly $\overline K$.

First let $a\in K$. Then for every $s$, the range of $\sigma^\smallfrown\langle a\rangle^\smallfrown\tau_s$ is contained in $K$, because $\operatorname{range}(\sigma)\subseteq K$, $a\in K$, and $\operatorname{range}(\tau_s)\subseteq K$. By the locking property of $\sigma$, we therefore have
\[
\learnerM(\sigma^\smallfrown\langle a\rangle^\smallfrown\tau_s)=\learnerM(\sigma)
\qquad\text{for all }s.
\]
So the procedure never prints $a$. Thus no element of $K$ is enumerated.

Now let $a\notin K$. Consider the infinite sequence
\[
\nu_a:=\sigma^\smallfrown\langle a\rangle^\smallfrown\kappa.
\]
This is an enumeration of the c.e.\ set $K\cup\{a\}$. Suppose towards a contradiction that the procedure never prints $a$. Then for every $s$,
\[
\learnerM(\sigma^\smallfrown\langle a\rangle^\smallfrown\tau_s)=\learnerM(\sigma).
\]
Hence, along the enumeration $\nu_a$, the learner stays constantly equal to $\learnerM(\sigma)$ after the point where $a$ is first seen. Since $\learnerM(\sigma)\mathrel{E^{\mathrm{ce}}_K}k$, the learner would converge on $\nu_a$ to a hypothesis in the $E^{\mathrm{ce}}_K$-class of $K$.

But $K\cup\{a\}\not\subseteq K$, while $K\subseteq K$, so the target $K\cup\{a\}$ is not $E^{\mathrm{ce}}_K$-equivalent to $K$. Therefore no correct hypothesis for the enumeration $\nu_a$ can lie in the $E^{\mathrm{ce}}_K$-class of $k$. This contradicts the assumption that $\learnerM$ computably $E^{\mathrm{ce}}_K$-learns all nonempty c.e.\ sets.

So when $a\notin K$, there must exist some $s$ such that
\[
\learnerM(\sigma^\smallfrown\langle a\rangle^\smallfrown\tau_s)\neq\learnerM(\sigma),
\]
and therefore the procedure eventually prints $a$.

We have shown that the procedure enumerates exactly $\overline K$. Thus $\overline K$ is c.e.\ Since $K$ is also c.e., the set $K$ would be computable, contradiction. Therefore $E^{\mathrm{ce}}_K$ is not computably universal.
\end{proof}

So even though universality admits a clean abstract characterisation, its effective counterpart is subtler: non-effective and computable tolerant learning do not coincide.

\section{Conclusion and open questions}

We introduced tolerant learning of c.e.\ sets: final hypotheses are required to be correct only up to a chosen equivalence relation on c.e.\ indices.  Comparing tolerance relations by the families they make learnable yields the tolerant learning hierarchy $\TLH$.  The hierarchy retains the central compactness obstruction of learning from positive data, as expressed by locking sequences, tell-tales, absorbing sets, and the locking-property characterisation of universality.  At the same time, it has a substantial degree-theoretic structure: it is bounded, is a lower semilattice, is not an upper semilattice, is atomless and coatomless, and has width and height $2^{\aleph_0}$.

The benchmark analysis reinforces the main message. Under learn-reducibility, the c.e.\ analogues of $E_0$, $E_1$, $E_3$, and $E_{\mathrm{set}}$ behave in ways that are not predicted by Borel, continuous, or computable reducibility.  In particular, $E^{\ce}_1$, $E^{\ce}_3$, $E^{\ce}_{\mathrm{set}}$, and $E^{\ce}_{\max}$ are pairwise incomparable, while $E^{\ce}_{\min}$ is universal.  The final separation involving $E^{\ce}_K$ shows that even universality itself bifurcates once computability of the learner is imposed.

\smallskip

Several questions remain open.
\begin{enumerate}
    \item Is $\TLH$ dense?  This is the most basic order-theoretic problem left by our analysis: whenever $E^{\ce}<_{\Learn}F^{\ce}$, must there be a degree strictly between them?
    \item Does every computably $E^{\mathrm{ce}}_0$-learnable family remain
computably $E^{\mathrm{ce}}_{\max}$-learnable? Our proof of the corresponding non-effective implication uses the noncomputable test of whether a conjectured c.e.\ set is infinite.
\end{enumerate}

These questions suggest that tolerant learning provides a setting in which semantic approximation, positive-data learning, and degree theory interact directly.

\end{document}

%% file: BorelRed.tikz
 \begin{tikzpicture}
 \node (bottom) at (0,0) {$=_1$};
  \node (Emin) at (0,1) {$E_{min},E_{max}$};
 \node (id) at (0,2) {=};
    \node (E0) at (0,3) {$E_0$};
     \node (E1) at (-1,4) {$E_1$};
      \node (E3) at (1,4) {$E_3$};
\node (Eset) at (1,5) {$E_{set}$};

    \path [-] (bottom) edge node {} (Emin);
    \path [-] (Emin) edge node {} (id);
\path [-] (id) edge node {} (E0);
     \path [-] (E0) edge node {} (E1);
     \path [-] (E0) edge node {} (E3);
     \path [-] (E3) edge node {} (Eset);

\end{tikzpicture}

%% file: contRed.tikz
 \begin{tikzpicture}
 \node (bottom) at (0,0) {$=_1$};
 \node (id) at (-1,1) {=};
 \node (Emax) at (1,1) {$E_{max}$};
 \node (Emin) at (0,1) {$E_{min}$};
    \node (E0) at (0,2) {$E_0$};
     \node (E1) at (-1,3) {$E_1$};
      \node (E3) at (1,3) {$E_3$};
\node (Eset) at (1,5) {$E_{set}$};

    \path [-] (bottom) edge node {} (id);
    \path [-] (bottom) edge node {} (Emax);
 \path [-] (bottom) edge node {} (Emin);
\path [-] (Emin) edge node {} (E0);
\path [-] (id) edge node {} (E0);
     \path [-] (E0) edge node {} (E1);
     \path [-] (E0) edge node {} (E3);
     \path [-] (Emax) edge node {} (E3);
      \path [-] (E3) edge node {} (Eset);

\end{tikzpicture}

%% file: EceHierarchy.tikz
 \begin{tikzpicture}
 \centering
\begin{scope}[every node/.style={thick}]

  \node (bottom) at (0,0) {$\Id_1$};
      \node (Emin) at (-1,1) {$E_{min}^{ce}$};
\node (Emax) at (1,1) {$E_{max}^{ce}$};
\node (id) at (0,2) {$=^{ce}$};
    \node (E0) at (0,3) {$E_0^{ce},E_1^{ce}$};
    \node (E3) at (0,4) {$E_3^{ce}$};
    \node (Eset) at (0,5) {$E_{set}^{ce}$};

\end{scope}

    \path [-] (bottom) edge node {} (Emin);
    \path [-] (bottom) edge node {} (Emax);

     \path [-] (Emin) edge node {} (id);
     \path [-] (Emax) edge node {} (id);
     \path [-] (id) edge node {} (E0);
     \path [-] (E0) edge node {} (E3);
     \path [-] (E3) edge node {} (Eset);

\end{tikzpicture}

%% file: TLH.tikz
 \begin{tikzpicture}
 \centering
\begin{scope}[every node/.style={thick}]

  \node (bottom) at (0,0) {$=^{ce}$};
    \node (E0) at (-1,1) {$E_0^{ce}$};
    \node (E1) at (-2,2.5) {$E_1^{ce}$};
    \node (E3) at (-1,2.5) {$E_3^{ce}$};
    \node (Eset) at (1.5,1) {$E_{set}^{ce}$};
    \node (Emax) at (0,2.5) {$E_{max}^{ce}$};
    \node (top) at (-1,4) {$\Id_1,E_{min}^{ce}$};

\end{scope}

    \path [-] (bottom) edge node {} (E0);
        \path [-] (bottom) edge node {} (Eset);

     \path [-] (E0) edge node {} (E1);
     \path [-] (E0) edge node {} (E3);
     \path [-] (E0) edge node {} (Emax);
     \path [-] (E1) edge node {} (top);
\path [-] (E3) edge node {} (top);
\draw [-] (Eset) to [bend right=30] (top);
\path [-] (Emax) edge node {} (top);

\end{tikzpicture}

%% file: manuscript.bbl
\bigskip

\begin{thebibliography}{BFKRSV24}

\bibitem[ABS17]{AndrewsBadaevSorbi2017Survey}
Uri Andrews, Serikzhan A. Badaev, and Andrea Sorbi.
\newblock A survey on universal computably enumerable equivalence relations.
\newblock In \emph{Computability and Complexity}, Lecture Notes in Computer Science, vol.~10010, pp.~418--451. Springer, Cham, 2017.

\bibitem[ABSM23]{andrews2023structure}
Uri Andrews, Daniel F. Belin, and Luca San Mauro.
\newblock On the structure of computable reducibility on equivalence relations of natural numbers.
\newblock \emph{Journal of Symbolic Logic} 88(3):1038--1063, 2023.

\bibitem[ASM24]{andrews24FSjump}
Uri Andrews and Luca San Mauro.
\newblock Investigating the computable {F}riedman--{S}tanley jump.
\newblock \emph{Journal of Symbolic Logic} 89(2):918--944, 2024.

\bibitem[Ang80]{angluin1980inductive}
Dana Angluin.
\newblock Inductive inference of formal languages from positive data.
\newblock \emph{Information and Control} 45(2):117--135, 1980.

\bibitem[BCJSS26]{BazhenovCiprianiJainSanMauroStephan2026Classifying}
Nikolay Bazhenov, Vittorio Cipriani, Sanjay Jain, Luca San Mauro, and Frank Stephan.
\newblock Classifying different criteria for learning algebraic structures.
\newblock \emph{Annals of Pure and Applied Logic} 177(1):103648, 2026.

\bibitem[BCSM22]{BazhenovCiprianiSanMauro2022MindChange}
Nikolay Bazhenov, Vittorio Cipriani, and Luca San Mauro.
\newblock Calculating the mind change complexity of learning algebraic structures.
\newblock In \emph{Revolutions and Revelations in Computability}, Lecture Notes in Computer Science, vol.~13359, pp.~1--12. Springer, Cham, 2022.

\bibitem[BCSM23]{BazhenovCiprianiSanMauro2023Borel}
Nikolay Bazhenov, Vittorio Cipriani, and Luca San Mauro.
\newblock Learning algebraic structures with the help of Borel equivalence relations.
\newblock \emph{Theoretical Computer Science} 951:113762, 2023.

\bibitem[BFKRSV24]{BazhenovFokinaRosseggerSoskovaVatev2024Text}
Nikolay Bazhenov, Ekaterina Fokina, Dino Rossegger, Alexandra Soskova, and Stefan Vatev.
\newblock Learning families of algebraic structures from text.
\newblock In \emph{Twenty Years of Theoretical and Practical Synergies}, Lecture Notes in Computer Science, vol.~14773, pp.~166--178. Springer, Cham, 2024.

\bibitem[BFSM20]{BazhenovFokinaSanMauro2020Informant}
Nikolay Bazhenov, Ekaterina Fokina, and Luca San Mauro.
\newblock Learning families of algebraic structures from informant.
\newblock \emph{Information and Computation} 275:104590, 2020.

\bibitem[BGJLS23]{BelangerGaoJainLiStephan2023}
David R. B\'elanger, Ziyuan Gao, Sanjay Jain, Wei Li, and Frank Stephan.
\newblock Learnability and positive equivalence relations.
\newblock \emph{Information and Computation} 295(Part~A):104913, 2023.

\bibitem[BSM21]{BazhenovSanMauro2021TuringComplexity}
Nikolay Bazhenov and Luca San Mauro.
\newblock On the Turing complexity of learning finite families of algebraic structures.
\newblock \emph{Journal of Logic and Computation} 31(7):1891--1900, 2021.

\bibitem[BB75]{blum1975toward}
Lenore Blum and Manuel Blum.
\newblock Toward a mathematical theory of inductive inference.
\newblock \emph{Information and Control} 28(2):125--155, 1975.

\bibitem[CL82]{case1982machine}
John Case and Christopher Lynes.
\newblock Machine inductive inference and language identification.
\newblock In \emph{Automata, Languages and Programming}, Lecture Notes in Computer Science, vol.~140, pp.~107--115. Springer, 1982.

\bibitem[CS78]{case1978anomaly}
John Case and Carl H. Smith.
\newblock Anomaly hierarchies of mechanized inductive inference.
\newblock In \emph{Proceedings of the 10th Annual ACM Symposium on Theory of Computing}, pp.~314--319. ACM, 1978.

\bibitem[CS83]{case1983comparison}
John Case and Carl H. Smith.
\newblock Comparison of identification criteria for machine inductive inference.
\newblock \emph{Theoretical Computer Science} 25(2):193--220, 1983.

\bibitem[CJS96]{caseJainSharma1996Anomalous}
John Case, Sanjay Jain, and Arun Sharma.
\newblock Anomalous learning helps succinctness.
\newblock \emph{Theoretical Computer Science} 164(1--2):13--28, 1996.

\bibitem[CHM12]{coskey2012hierarchy}
Samuel Coskey, Joel David Hamkins, and Russell Miller.
\newblock The hierarchy of equivalence relations on the natural numbers under computable reducibility.
\newblock \emph{Computability} 1(1):15--38, 2012.

\bibitem[Fel72]{feldman1972some}
Jerome Feldman.
\newblock Some decidability results on grammatical inference and complexity.
\newblock \emph{Information and Control} 20(3):244--262, 1972.

\bibitem[FFN12]{FokinaFriedmanNies2012Sigma3}
Ekaterina B. Fokina, Sy-David Friedman, and Andr\'{e} Nies.
\newblock Equivalence relations that are $\Sigma^0_3$-complete for computable reducibility.
\newblock In \emph{Logic, Language, Information and Computation (WoLLIC 2012)}, Lecture Notes in Computer Science, vol.~7456, pp.~26--33. Springer, Berlin, Heidelberg, 2012.
\newblock \href{https://doi.org/10.1007/978-3-642-32621-9_2}{doi:10.1007/978-3-642-32621-9\_2}.

\bibitem[FKSM19]{FokinaKotzingSanMauro2019EquivalenceStructures}
Ekaterina Fokina, Timo K\"otzing, and Luca San Mauro.
\newblock Limit learning equivalence structures.
\newblock In \emph{Proceedings of the 30th International Conference on Algorithmic Learning Theory}, Proceedings of Machine Learning Research, vol.~98, pp.~383--403. PMLR, 2019.

\bibitem[Gao09]{gao2008invariant}
Su Gao.
\newblock \emph{Invariant Descriptive Set Theory}.
\newblock CRC Press, Boca Raton, FL, 2009.

\bibitem[GG01]{gao2001computably}
Su Gao and Peter Gerdes.
\newblock Computably enumerable equivalence relations.
\newblock \emph{Studia Logica} 67(1):27--59, 2001.

\bibitem[Gol67]{Gold67}
E. Mark Gold.
\newblock Language identification in the limit.
\newblock \emph{Information and Control} 10(5):447--474, 1967.

\bibitem[HKL90]{harringtonKechrisLouveau1990GlimmEffros}
Leo A. Harrington, Alexander S. Kechris, and Alain Louveau.
\newblock A Glimm--Effros dichotomy for Borel equivalence relations.
\newblock \emph{Journal of the American Mathematical Society} 3(4):903--928, 1990.

\bibitem[HS07]{HarizanovStephan2007VectorSpaces}
Valentina S. Harizanov and Frank Stephan.
\newblock On the learnability of vector spaces.
\newblock \emph{Journal of Computer and System Sciences} 73(1):109--122, 2007.

\bibitem[JORS99]{jain1999systems}
Sanjay Jain, Daniel Osherson, James Royer, and Arun Sharma.
\newblock \emph{Systems That Learn: An Introduction to Learning Theory}.
\newblock 2nd ed., MIT Press, 1999.

\bibitem[JS03]{JainStephan2003TourRobust}
Sanjay Jain and Frank Stephan.
\newblock A tour of robust learning.
\newblock In \emph{Computability and Models}, pp.~215--247. Springer, Boston, MA, 2003.

\bibitem[LGZ05]{LangeGrieserZeugmann2005Approx}
Steffen Lange, Gunter Grieser, and Thomas Zeugmann.
\newblock Inductive inference of approximations for recursive concepts.
\newblock \emph{Theoretical Computer Science} 348(1):15--40, 2005.

\bibitem[LZZ08]{LangeZeugmannZilles2008Survey}
Steffen Lange, Thomas Zeugmann, and Sandra Zilles.
\newblock Learning indexed families of recursive languages from positive data: a survey.
\newblock \emph{Theoretical Computer Science} 397(1--3):194--232, 2008.

\bibitem[MS04]{MerkleStephan2004Trees}
Wolfgang Merkle and Frank Stephan.
\newblock Trees and learning.
\newblock \emph{Journal of Computer and System Sciences} 68(1):134--156, 2004.

\bibitem[OW82]{osherson1982criteria}
Daniel N. Osherson and Scott Weinstein.
\newblock Criteria of language learning.
\newblock \emph{Information and Control} 52(2):123--138, 1982.

\bibitem[Put65]{putnam1965trial}
Hilary Putnam.
\newblock Trial and error predicates and the solution to a problem of Mostowski.
\newblock \emph{Journal of Symbolic Logic} 30(1):49--57, 1965.

\bibitem[Roy86]{royer1986Approximations}
James S. Royer.
\newblock Inductive inference of approximations.
\newblock \emph{Information and Control} 70(2--3):156--178, 1986.

\bibitem[RSS25]{RosseggerSlamanSteifer2025Polish}
Dino Rossegger, Theodore Slaman, and Tomasz Steifer.
\newblock Learning equivalence relations on Polish spaces.
\newblock \emph{The Journal of Symbolic Logic}, First View, pp.~1--19, 2025.
\newblock \href{https://doi.org/10.1017/jsl.2025.7}{doi:10.1017/jsl.2025.7}.

\bibitem[SV90]{smithVelauthapillai1990ApproximatePrograms}
Carl H. Smith and Mahendran Velauthapillai.
\newblock On the inference of approximate programs.
\newblock \emph{Theoretical Computer Science} 77(3):249--266, 1990.

\bibitem[SV01]{StephanVentsov2001LearningText}
Frank Stephan and Yuri Ventsov.
\newblock Learning algebraic structures from text.
\newblock \emph{Theoretical Computer Science} 268(2):221--273, 2001.

\bibitem[Wha74]{wharton1974approximate}
R. Michael Wharton.
\newblock Approximate language identification.
\newblock \emph{Information and Control} 26(3):236--255, 1974.

\end{thebibliography}
